\pdfoutput=1
\documentclass[12pt]{article}

\usepackage{lmodern}

\DeclareFontFamily{T1}{calligra}{}
\DeclareFontShape{T1}{calligra}{m}{n}{<->s*[1.44]callig15}{}
\DeclareMathAlphabet\mathcalligra   {T1}{calligra} {m} {n}
\DeclareMathAlphabet\mathzapf       {T1}{pzc} {mb} {it}
\DeclareMathAlphabet\mathchorus     {T1}{qzc} {m} {n}
\DeclareMathAlphabet\mathrsfso      {U}{rsfso}{m}{n}
\DeclareMathAlphabet\mathfrcal      {T1}{frcursive}{m}{it}
\DeclareFontFamily{T1}{frcursive}{}
\DeclareFontShape{T1}{frcursive}{m}{n}{<->s*[1.44]callig15}{}
\DeclareMathAlphabet\mathfrcal      {T1}{frcursive}{m}{it}

\usepackage{amsmath}
\usepackage{amssymb}
\usepackage{amsthm}
\usepackage{graphicx}
\usepackage[dvipsnames]{xcolor}
\usepackage[english]{babel}
\usepackage{csquotes}
\numberwithin{equation}{section}
\usepackage{array}
\usepackage{mathtools}
\usepackage{dsfont}
\usepackage{mathrsfs}
\usepackage{tikz}\usetikzlibrary{matrix,fit}
\usepackage{varwidth}
\usepackage{enumerate}
 \usepackage{appendix}
\usepackage{ytableau}
\usepackage{simpler-wick}
\usepackage{dsfont}

\usepackage{caption}

\usepackage[compat=1.1.0]{tikz-feynman}

\usepackage[margin=1in]{geometry}
\usepackage{nicematrix}

\usepackage[
    backend=bibtex,
    style=alphabetic,
maxbibnames=99,
giveninits=true,
minalphanames=1,
maxalphanames=3,url=false
  ]{biblatex}
\usepackage{mathabx}
\usepackage{empheq}

\usepackage{setspace}

\usepackage{fontawesome5}

\usepackage{slashed}
\usepackage{upgreek}

\usepackage{tabstackengine}

\usepackage{wrapfig}
\usepackage[abs]{overpic}

\usepackage{float}

\usepackage{mathtools}

\fixTABwidth{T}

\usepackage{accents}

\newcommand{\CP}{\mathds{CP}}
\newcommand{\CC}{\mathds{C}}

\renewcommand{\tilde}{\widetilde}

\newcommand{\tSO}{\mathrm{SO}}
\newcommand{\tU}{\mathrm{U}}

\newcommand{\bea}{\begin{equation}}
\newcommand{\eea}{\end{equation}}
\newcommand{\bear}{\begin{eqnarray}}
\newcommand{\eear}{\end{eqnarray}}
\newcommand{\bearr}{\begin{eqnarray*}}
\newcommand{\eearr}{\end{eqnarray*}}

\newtheorem{prop}{Proposition}
\newtheorem{lem}{Lemma}
\newtheorem{corol}{Corollary}
\newtheorem{examp}{Example}

\newdimen\mytextwidth
\newcommand\rem[2][cyan!40!green]{\noindent\nobreak\hfil\penalty1000\hfilneg
\mytextwidth=\linewidth\advance\mytextwidth by 2mm
\begin{tikzpicture}[baseline=-\the\dimexpr\fontdimen22\textfont2\relax]\node[outer sep=0pt,draw=black,fill=#1,fill opacity=1,text opacity=1,rectangle,rounded corners]{\begin{varwidth}{\mytextwidth}\textcolor{white}{#2}\end{varwidth}};
\end{tikzpicture}\allowbreak
}

\newcommand\whiterem[2][white!]{\noindent\nobreak\hfil\penalty1000\hfilneg
\mytextwidth=\linewidth\advance\mytextwidth by 2mm
\begin{tikzpicture}[baseline=-\the\dimexpr\fontdimen22\textfont2\relax]\node[outer sep=0pt,draw=black,fill=#1,fill opacity=1,text opacity=1,rectangle,rounded corners,line width=1.5pt]{\begin{varwidth}{\mytextwidth}\textcolor{black}{#2}\end{varwidth}};
\end{tikzpicture}\allowbreak
}

\newcommand{\SU}{\mathsf{SU}}

\renewbibmacro{in:}{}

\usepackage{mdframed}

\newbibmacro{string+doiurlisbn}[1]{%
  \iffieldundef{doi}{%
    \iffieldundef{url}{%
      \iffieldundef{isbn}{%
        \iffieldundef{issn}{%
          #1%
        }{%
          \href{http://books.google.com/books?vid=ISSN\thefield{issn}}{#1}%
        }%
      }{%
        \href{http://books.google.com/books?vid=ISBN\thefield{isbn}}{#1}%
      }%
    }{%
      \href{\thefield{url}}{#1}%
    }%
  }{%
    \href{http://dx.doi.org/\thefield{doi}}{#1}%
  }%
}

\DeclareFieldFormat{title}{\usebibmacro{string+doiurlisbn}{\mkbibemph{#1}}}
\DeclareFieldFormat[article,incollection]{title}%
    {\usebibmacro{string+doiurlisbn}{\mkbibquote{#1}}}

\newmdenv[
  topline=false,
  bottomline=false,
  rightline=false,
  linewidth=2pt,
  skipabove=\topsep,
  skipbelow=\topsep
]{siderules}

\newmdenv[
  topline=false,
  bottomline=false,
  linewidth=2pt,
  skipabove=\topsep,
  skipbelow=\topsep
]{siderulesright}

\makeatletter
\renewcommand{\@seccntformat}[1]{\csname the#1\endcsname.\quad}
\makeatother

\usepackage{setspace}
\usepackage{xpatch}

\makeatletter
\renewcommand{\@chap@pppage}{
  \clear@ppage
  \thispagestyle{plain}
  \if@twocolumn\onecolumn\@tempswatrue\else\@tempswafalse\fi
  \null\vfil
  \markboth{}{}
  {\centering
   \interlinepenalty \@M
   \normalfont
   \MakeUppercase \appendixpagename\par}
  \if@dotoc@pp
    \addappheadtotoc
  \fi
  \vfil\newpage
  \if@twoside
    \if@openright
      \null
      \thispagestyle{empty}
      \newpage
    \fi
  \fi
  \if@tempswa
    \twocolumn
  \fi
}
\makeatother

\definecolor{navycol}{RGB}{100,150,160}
   \definecolor{pinkcol}{RGB}{242,55,55}
   \definecolor{greencol}{RGB}{50,205,50}

   \definecolor{bluecol}{RGB}{30,144,255}

\usepackage{tocloft}

\usepackage{titlesec}

\titleformat*{\section}{\large\bfseries}
\titleformat*{\subsection}{\normalsize\bfseries}
\titleformat*{\subsubsection}{\normalsize\bfseries}
\titleformat*{\paragraph}{\large\bfseries}
\titleformat*{\subparagraph}{\large\bfseries}
\titlespacing{\author}{-5pt}{-5pt}{-5pt}[-5pt]

\makeatletter
\renewcommand\subsubsection{\@startsection{subsubsection}{3}{\z@}
                                     {-3.25ex\@plus -1ex \@minus -.2ex}
                                     {-1.5ex \@plus -.2ex}
                                     {\normalfont\normalsize\bfseries}}
\renewcommand\subsection{\@startsection{subsection}{3}{\z@}
                                     {-3.25ex\@plus -1ex \@minus -.2ex}
                                     {-1.5ex \@plus -.2ex}
                                     {\normalfont\normalsize\bfseries}}                                     
\makeatother

\usepackage{soul}

\usepackage{scalerel}[2016/12/29]

\DeclareFontFamily{U}{solomos}{}
\DeclareErrorFont{U}{solomos}{m}{n}{10}
\DeclareFontShape{U}{solomos}{m}{n}{
  <-> s*[1.1]  gsolomos8r
}{}

   \usepackage{stmaryrd}

\usepackage{tikz}
\usetikzlibrary{arrows.meta}

\usepackage{indentfirst}

\let \savenumberline \numberline
\def \numberline#1{\savenumberline{#1.}}

\usepackage{xurl}
\usepackage{scrextend}

\usepackage{upgreek}

\usepackage{etoolbox}
\patchcmd{\tableofcontents}{\@starttoc}{\vspace{-0.3cm}\@starttoc}{}{}

\usepackage{nicematrix}

\newcounter{Chapcounter}

\newcommand{\chapter}[1] 
{ {\centering          
  \addtocounter{Chapcounter}{1} \Large \underline{\sffamily \texorpdfstring{\textbf{  Chapter \theChapcounter: ~#1}}{Lg}} }   
  \addcontentsline{toc}{section}{ \color{blue} \texorpdfstring{Chapter ~}{Lg}\theChapcounter.\texorpdfstring{~~}{Lg} #1 }    
}

\newcommand{\appendixbig}[1] 
{ {\centering          
   \Large \underline{\sffamily \textbf{  Appendices}} }    
}

\usepackage[hyperfootnotes=]{hyperref}
\hypersetup{
    colorlinks=true,
    linkcolor=blue,
    filecolor=magenta,      
    urlcolor=cyan,
    breaklinks=true
    }

\title{Isotropic embeddings of coadjoint orbits \\and magnetic geodesic flows}

\author{Dmitri Bykov$^{\,a,\,b,\,c,\,d}$\footnote{Emails:
 bykov@mi-ras.ru, dmitri.v.bykov@gmail.com} \,,  Andrew Kuzovchikov$^{\,a,\,b}$\footnote{Emails: kuzovchikov@mi-ras.ru, 
 andrkuzovchikov@mail.ru}
\\  \vspace{0cm}  \\
{\small $a)$ \emph{Steklov
Mathematical Institute of Russian Academy of Sciences,}} \\{\small \emph{Gubkina str. 8, 119991 Moscow, Russia} }\\
{\small $b)$ \emph{Institute for Theoretical and Mathematical Physics,}} \\{\small \emph{Lomonosov Moscow State University, 119991 Moscow, Russia}} \\
{\small $c)$ \emph{Moscow Institute of Physics and Technology,}} \\
{\small \emph{Institutskii per. 9, 141702  Dolgoprudny, Russia}}
\\
{\small $d)$ \emph{Beijing Institute of Mathematical Sciences and Applications (BIMSA),}} \\{\small \emph{Huairou District, Beijing
101408, China}} 
}

\date{}

\begin{document}

\maketitle

\ytableausetup{centertableaux}

\vspace{0cm}
\textbf{Abstract.} We consider isotropic and Lagrangian embeddings of coadjoint orbits of compact Lie groups into products of coadjoint orbits. After reviewing the known facts in the case of $\mathrm{SU}(n)$ we initiate a similar study for $\mathrm{SO}$ and $\mathrm{Sp}$ cases. In the second part we apply this to the study of dynamical systems with $\mathrm{SU}(n)$ symmetry, proving equivalence between systems of two types: those describing magnetic geodesic flow on flag manifolds and classical `spin chains' of a special type.

\newpage
\tableofcontents

\section{Introduction}

Coadjoint orbits of classical compact Lie groups -- $\mathrm{SU}(n)$, $\tSO(n)$ and $\mathrm{Sp}(n)$ -- are remarkable homogeneous spaces called generalized flag manifolds. These are complex, symplectic, K\"ahler manifolds, cf.~\cite{BordemannForger}. In the present paper we study isotropic (and sometimes Lagrangian) embeddings of such flag manifolds into products of other flag manifolds.

It is a beautiful geometric fact that such embeddings exist, which serves as a mathematical underpinning of the so-called Haldane mapping between spin chains and sigma models. This mapping was first established in the case of the $\CP^1$ model in~\cite{HaldaneNonlin} and was quoted as one of the official reasons for Haldane's Nobel Prize award~\cite{HaldaneNobel}. The relation of his discovery to symplectic geometry was in turn explained in~\cite{BykLagEmb, Bykov:2012am} (see also the review~\cite{AffleckReview}) and used to construct generalizations to the flag manifolds of $\mathrm{SU}(n)$.

In the present paper we review the subsequent developments~\cite{Bykov_2024}, emphasizing that at least in the $\mathrm{SU}(n)$ case the statements can be strengthened to describe not only the vicinity of the respective Lagrangian submanifold, but rather a whole dense set in the ambient space. This can be used to prove an equivalence between two seemingly unrelated dynamical systems: one describing interactions of a number of classical `spins', the other corresponding to geodesic motion on a flag manifold, possibly in the background of a monopole-like magnetic field.

Our second goal is to take the first steps towards an extension of the theory to the case of $\tSO$- and $\mathrm{Sp}$- orbits. This turns out to be not as straightforward as one might naively expect. In this case we are able to prove that there exists an \emph{isotropic} embedding of an arbitrary generalized  flag manifold into a product of Grassmannians, but, in general, this embedding is not Lagrangian. However, in the $\tSO(2n)$ case we do find a series of two-step flag manifolds that are Lagrangian in a product of Grassmannians. One should emphasize that our construction does not exclude the possibility that there exist other, so far unknown, Lagrangian embeddings.

\section{Generalized flag manifolds}
In this section, we recall some geometric properties of generalized flag manifolds. 
The generalized flag manifold $\mathcal{O}_{\Uplambda}^{\,\mathrm{G}}$  is a (co)adjoint orbit of the compact semisimple Lie group $\mathrm{G}$
\begin{align}
    \mathcal{O}_{\Uplambda}^{\,\mathrm{G}} := \{\;\mathrm{Ad}_{g}\Uplambda \,,\; g \in \mathrm{G}\;\}\,,
\end{align}
where $\Uplambda$ is an element of the Lie algebra $\mathfrak{g}$ of the group $\mathrm{G}$, $\mathrm{Ad}$ denotes adjoint action\footnote{For a matrix Lie group $\mathrm{Ad}_{g} x$ is simply $g x g^{-1}$} of $\mathrm{G}$. In general, every (co)adjoint orbit intersects\footnote{We refer to the book \cite{bump2004lie} for details on Lie groups.} with a maximal torus in $\mathrm{G}$. Thus, we will always assume that $\Uplambda$ is an element of the Cartan subalgebra  
 $\mathfrak{t}$ of $\mathfrak{g}$. 

The full classification of generalized flag manifolds has been done via painted Dynkin diagram technique and can be found, for example, in~\cite{BordemannForger, AlekseevskyFlags}. Any (co)adjoint orbit is a symplectic manifold \cite{KirillovMerits},  endowed with the Kirillov-Kostant-Souriau symplectic form 
\begin{align}
    \omega_x (\xi,\zeta) := \langle x, \left[\xi,\zeta\right]\rangle,
\end{align}
where $\langle \bullet, \bullet \rangle$ is a Killing form\footnote{For a matrix Lie group $\langle X,Y \rangle = \mathrm{Tr}\left(XY\right)$.} on $\mathfrak{g}$; $x \in \mathcal{O}_{\Uplambda}^{\,\mathrm{G}}$; $\xi,\zeta\, \in\, \mathrm{T}_{x}\mathcal{O}_{\Uplambda}^{\,\mathrm{G}}$. For $\alpha \in \mathbb{R} \backslash \{0\}$, $\Uplambda$ and $\alpha \cdot\Uplambda$ define  diffeomorphic generalized flag manifolds with a different normalization of the symplectic structure.  

Clearly, $\mathcal{O}_{\Uplambda}^{\,\mathrm{G}}$ is a homogeneous space of $\mathrm{G}$. Therefore, we can present the generalized flag manifold as a quotient space $\mathrm{G}/\mathrm{Stab}_{\Uplambda}$, where $\mathrm{Stab}_{\Uplambda} \subset \mathrm{G}$ is the  stabilizer of $\Uplambda \in \mathfrak{t}$. The  moment map $\mu:\,\mathcal{O}_{\Uplambda}^{\,\mathrm{G}} \mapsto \mathfrak{g}^\ast$ corresponding to the action of $\mathrm{G}$ is given by
\begin{align}
    \mu(x) = \langle x,\bullet \rangle\,.
\end{align}

The well-known examples of flag manifolds are  the $n$-dimensional projective space~$\mathbb{CP}^n$,  the Grassmannian $\mathrm{Gr}(k, n)$ of $k$-dimensional subspaces in $\mathbb{C}^n$, etc. 

\subsection{Isotropic and Lagrangian embeddings.} \label{Embeddings}
In this section we will describe  embeddings of generalized flag manifolds into products of other generalized flag manifolds. The embeddings will always be isotropic and sometimes Lagrangian (for other constructions of Lagrangian submanifolds in homogeneous spaces cf.~\cite{Mironov2004,Tyurin2021,NoharaUeda}).

Consider the orbit $\mathcal{O}_{\Uplambda}^{\,\mathrm{G}}$ of a group $\mathrm{G}$ and assume the following decomposition of~$\Uplambda \in \mathfrak{t}$ (for some positive integer $r$)
\begin{align}\label{Lambda1}
    \Uplambda = \sum_{i=0}^r \mathrm{C}_i \Uplambda_{i}\,,
\end{align}
where $\mathrm{C}_i \in \mathbb{R}$ and $\Uplambda_{i} \in \mathfrak{t}$ are such that 
\begin{align}\label{Lambdasum}
    \sum_{i=0}^r \Uplambda_{i} = 0\,.
\end{align}
Next, consider the map \begin{align}\label{orbitmap}
\mathcal{M}:\,\mathcal{O}_{\Uplambda}^{\,\mathrm{G}} \mapsto \prod\limits_{i=0}^r\mathcal{O}_{\Uplambda_i}^{\,\mathrm{G}}\,,\quad\quad \mathrm{Ad}_{g}\Uplambda \mapsto \{\mathrm{Ad}_{g}\Uplambda_i\}_{i=0}^r\,.
\end{align}
Since in the l.h.s. $g\in\mathrm{G} / \text{Stab}\, \Uplambda $ is defined up to right multiplication by an element of $\text{Stab}\, \Uplambda$, whereas in each term in the r.h.s. $g\in\mathrm{G} / \text{Stab}\, \Uplambda_i$, in order for~(\ref{orbitmap}) to be a well-defined map, one should require that 
$
    \mathrm{Stab}_\Uplambda \subset  \mathrm{H}_0:=\bigcap_{i=0}^r \mathrm{Stab}_{\Uplambda_i}\,.
$ 
In turn, if one requires that $\mathcal{M}$ be an embedding, one should impose  
\begin{align}\label{LambdaStab}
\mathrm{Stab}_\Uplambda = \bigcap_{i=0}^r \mathrm{Stab}_{\Uplambda_i}\,.
\end{align}
Indeed, in general  $\mathcal{M}$ can be written as a composition
\begin{align}
\mathcal{M}:\,\mathcal{O}_{\Uplambda}^{\,\mathrm{G}}\simeq \frac{\mathrm{G}}{\mathrm{Stab}_\Uplambda} \overset{\pi}{\mapsto} \frac{\mathrm{G}}{\mathrm{H}_0} \overset{i}{\hookrightarrow} \prod\limits_{i=0}^r\mathcal{O}_{\Uplambda_i}^{\,\mathrm{G}}\,,
\end{align}
with $\pi$ a forgetful map (with fiber $\mathrm{H}_0/ \mathrm{Stab}\,\Uplambda$) and $i$ an embedding. 

We choose the symplectic form $\Omega_{\mathrm{prod}}$ on $\prod\limits_{i=0}^r\mathcal{O}_{\Uplambda_i}^{\,\mathrm{G}}$ in a canonical way, which is to say, as a sum of the Kirillov-Kostant-Souriau forms on each $\mathcal{O}_{\Uplambda_i}^{\,\mathrm{G}}$. Consider the diagonal action of $\mathrm{G}$ on $\prod\limits_{i=0}^r\mathcal{O}_{\Uplambda_i}^{\,\mathrm{G}}$. The corresponding moment map is simply 
\begin{align}
    \Tilde{\mu}(x) = \sum_{i=0}^{r} \;\langle x_i,\bullet \rangle\,,\quad\quad\textrm{where}\quad\quad x_i \in \mathcal{O}_{\Uplambda_i}^{\,\mathrm{G}}\,.
\end{align}

The image of $\mathcal{O}_{\Uplambda}^{\,\mathrm{G}}$ under the map $\mathcal{M}$ lies in $\Tilde{\mu}^{-1}(0)$, which follows immediately from the linearity of the Killing form and~(\ref{Lambdasum}). In fact, the zero level of the moment map is $\Tilde{\mu}^{-1}(0) = \{\,\sum_{i=0}^{r}\,x_i = 0 \,\}$. Since~$\mathcal{O}_{\Uplambda}^{\,\mathrm{G}}$ is a $\mathrm{G}$-orbit, it is an isotropic submanifold of $\prod\limits_{i=0}^r\mathcal{O}_{\Uplambda_i}^{\,\mathrm{G}}$. Indeed, for any two vector fields $\xi, \zeta$ along the group orbit one has
\begin{align} \label{orbitisotropic}
\iota_{\xi}\,\iota_{\zeta}\,\Omega_{\mathrm{prod}} \big|_{\mathcal{O}_{\Uplambda}^{\,\mathrm{G}}} = \nabla_\xi\,\Tilde{\mu}[\zeta]\big|_{\mathcal{O}_{\Uplambda}^{\,\mathrm{G}}} = 0 \,,   
\end{align} 
where $\Tilde{\mu}[\zeta]$ stands for the natural pairing with a Lie algebra element $\zeta$ (identified with the respective vector field). 
Since  $\mathcal{O}_{\Uplambda}^{\,\mathrm{G}}$ is a $\mathrm{G}$-orbit, its tangent space at a point is isomorphic to a subspace of $\mathfrak{g}$, hence  $\Omega_{\mathrm{prod}} \big|_{\mathcal{O}_{\Uplambda}^{\,\mathrm{G}}} = 0$.

We have thus proven the following proposition:
\begin{prop} \label{embeddingProp}
    Under the conditions~(\ref{Lambda1}), (\ref{Lambdasum}), (\ref{LambdaStab}), $\mathcal{O}_{\Uplambda}^{\,\mathrm{G}}$ is an isotropic submanifold of $\prod\limits_{i=0}^r\mathcal{O}_{\Uplambda_i}^{\,\mathrm{G}}$, embedded via the map~(\ref{orbitmap}).
\end{prop}

To see whether this embedding is Lagrangian it suffices to calculate the dimensions:  $\mathcal{O}_{\Uplambda}^{\,\mathrm{G}}$ is a Lagrangian submanifold of  $\prod\limits_{i=0}^r\mathcal{O}_{\Uplambda_i}^{\,\mathrm{G}}$ if and only if
\begin{align}
    2\,\mathrm{dim}\left(\mathcal{O}_{\Uplambda}^{\,\mathrm{G}}\right) = \sum\limits_{i=0}^r\,\mathrm{dim}\left(\mathcal{O}_{\Uplambda_i}^{\,\mathrm{G}}\right)\,.
\end{align}
We instantly deduce the following elementary consequence as a special case: 
\begin{corol} \label{simpleLagrExample}
    The embedding $\mathcal{O}_{\Uplambda}^{\,\mathrm{G}} \hookrightarrow \mathcal{O}_{\Uplambda}^{\,\mathrm{G}} \times \mathcal{O}_{-\Uplambda}^{\,\mathrm{G}}$ is Lagrangian.
\end{corol}
\begin{proof}
    This follows from Proposition \ref{embeddingProp} by taking  $\Uplambda_0 = \Uplambda$, $\Uplambda_1 = -\Uplambda$ and $\mathrm{C}_0=1/2$, $\mathrm{C}_1=-1/2$. 
\end{proof}

\section{$\mathrm{SU}(n)$ case}
In this section, we will discuss  examples of generalized flag manifolds of  the $\mathrm{SU}(n)$ group. Here the Cartan subalgebra $\mathfrak{t}$ consists of the matrices 
\begin{align}
    \Uplambda = 
    \begin{pmatrix}
        \lambda_0 \cdot \mathds{1}_{n_0}  & & \\
        & \lambda_1 \cdot \mathds{1}_{n_1} & \\
        & & \ddots \\ 
        & & &  \lambda_r \cdot \mathds{1}_{n_r}
    \end{pmatrix}\,,
\end{align}
where $\mathds{1}_{n}$ denotes the $(n\times n)$ identity matrix, the eigenvalues $\lambda_i$ are distinct real numbers, $n_0+n_1+\dots+n_r = n$ and $\mathrm{Tr} \left(\Uplambda \right)=n_0\lambda_0+\dots + n_r \lambda_r = 0$. The orbit of $\Uplambda$ is  
\begin{align}
    \mathcal{F}_{n_0,n_1,\dots,n_r} := \mathcal{O}_{\Uplambda}^{\,\mathrm{SU}(n)} = \frac{\mathrm{SU}(n)}{\mathrm{S}\left(\mathrm{U}(n_0)\times\dots\times\mathrm{U}(n_r)\right)}\,. \label{defFlagManifold}
\end{align}
We will henceforth use the notation $\mathcal{F}_{n_0, n_1, \dots, n_r}$ to denote the general $\mathrm{SU}(n)$ flag manifold.

There is another, more geometric, description of the orbit~(\ref{defFlagManifold}) that explains its name, the flag manifold. Pick an ordering of the integers $n_0, \ldots, n_r$ featuring in the denominator. Here for simplicity we will assume that it is simply the lexicographic ordering fixed by the subscripts of $n_i$. Define the \emph{flag} of embedded linear spaces
\begin{align}\label{flagofspaces}
    0\subset L_1 \subset \cdots \subset L_{r+1}\simeq \CC^n\,,\quad\quad \textrm{where}\quad\quad \mathrm{dim}\,L_k=\sum_{j=1}^{k}\,n_{j-1}
\end{align}
Then $\mathcal{F}_{n_0, n_1, \dots, n_r}$ may be thought of as the manifold of such flags. Notice that the ordering of the $n_i$'s determines the dimensions of the nested subspaces. In fact, this is tantamount to the choice of an invariant complex structure on~$\mathcal{F}_{n_0, n_1, \dots, n_r}$ (cf.~\cite{BorelHirzebruch, AlekseevskyFlags} or the review~\cite{AffleckReview}).

Now we are ready to construct the embedding via the following set of $\{\Uplambda_i\}_{i=0}^r$:
\begin{align}\label{Lambda U(n) Grass}
    \Uplambda_i = 
    \begin{pmatrix}
    -n_i \cdot \mathds{1}_{d_{i-1}}  & & \\
     & (n-n_i) \cdot \mathds{1}_{n_i} & \\
     & & -n_i \cdot \mathds{1}_{n - d_{i}}
    \end{pmatrix} \,,
\end{align}
where $d_i = \sum_{k=0}^{i}\,n_k$ and $d_{-1}\equiv 0$. 

The orbit of each $\Uplambda_i$ is a Grassmann manifold 
\begin{align}
    \mathrm{Gr}(n_i, n) :=  \mathcal{O}_{\Uplambda_i}^{\,\mathrm{SU}(n)} = \frac{\mathrm{SU}(n)}{\mathrm{S}\left(\mathrm{U}(n_i)\times\mathrm{U}(n - n_i)\right)}
\end{align}
with the standard Fubini-Study symplectic form. When $n_i = 1$, we obtain the $(n-1)$-dimensional projective space $\mathbb{CP}^{n-1}$. 

Following Section \ref{Embeddings} and Proposition \ref{embeddingProp}, we get the isotropic embedding
\begin{align}\label{flagEmbedding}
    \mathcal{F}_{n_0,n_1,\dots,n_r} \hookrightarrow \mathrm{Gr}(n_0, n) \times \mathrm{Gr}(n_1, n) \times \dots \times \mathrm{Gr}(n_r, n)\,.
\end{align}
A simple calculation of the dimensions shows that the embedding is Lagrangian. Thus, we have proved the following proposition 
\begin{prop}[\cite{Bykov:2012am}]\label{lagrtheorem}
    $\mathcal{F}_{n_0,n_1,\dots,n_r}$ is a Lagrangian submanifold of $\left(\mathrm{Gr}(n_0, n) \times \dots \times \mathrm{Gr}(n_r, n), p\,\Omega\right)$, where the symplectic form $\Omega$ is a sum of Fubini-Study forms and $p\in \mathbb{Z}$ an overall factor. 
\end{prop}
The proposition also means that at least a small neighborhood of $\mathcal{F}_{n_0,n_1,\dots,n_r}$ in the product of Grassmannians is symplectomorphic to $\mathrm{T}^\ast \mathcal{F}_{n_0,n_1,\dots,n_r}$ \cite{WEINSTEINLagr}. In fact, we will soon prove a refined version of this fact, but let us setup some notation first.

It follows from the definition (\ref{defFlagManifold}) that $\mathcal{F}_{n_0,n_1,\dots,n_r}$ may be parameterized by a set of mutually orthogonal planes, each plane of dimension $n_i$. Moreover, within each plane we may choose an orthonormal basis of vectors $\{\,z_a \in \mathbb{CP}^{n-1}\,\}_{a=d_{i-1}+1}^{d_i}$.   
Vectors within each such plane are defined up to an $\mathrm{U}(n_i)$ transformation. For example, in the case of $\mathrm{Gr}(k,n)$ one needs  a single set of vectors. Even more simply, $\mathbb{CP}^{n-1}$ can be defined by a single vector. We will always assume that all vectors $z_a$ are unit-normalized, i.e. $|z_a|^2 = 1$.

Consider the Grassmannian $\mathrm{Gr}(n_k,n)$: we group the corresponding vectors in a single matrix 
\begin{align}
    \mathsf{Z}_k = \begin{pmatrix}
    z_{d_{k-1}+1} & z_{d_{k-1}+2} & \dots & z_{d_k}
\end{pmatrix}\,.
\end{align} 
The standard Fubini-Study form on $\mathrm{Gr}(n_k,n)$ is then given by
\begin{align}
    \Omega_{\mathrm{Gr}(n_k,n)} = i\, \mathrm{Tr} \left(\mathrm{d} \mathsf{Z}_k^{\dagger} \wedge \mathrm{d} \mathsf{Z}_k\right). 
\end{align}
The Lagrangian embedding (\ref{flagEmbedding}) is defined by imposing the orthogonality condition on the vectors parametrizing different Grassmannians in the r.h.s. of (\ref{flagEmbedding}). 

Now we will take a closer look at $\mathrm{Gr}(n_0, n) \times \dots \times \mathrm{Gr}(n_r, n)$. To start with, we define a matrix 
\begin{align}
    \mathsf{Z} = 
    \begin{pmatrix}
        \mathsf{Z}_0 & \mathsf{Z}_1 & \dots & \mathsf{Z}_r
    \end{pmatrix}
\end{align}
parametrizing elements of the product of Grassmannians. We  define the determinantal variety  $\mathcal{D} := \{\,\mathrm{det}\left(\mathsf{Z}\right) = 0\,\}$ and the open set 
\begin{align}
\mathcal{X} = \prod_{i=0}^r \mathrm{Gr}(n_i, n)\,\Bigr\backslash\, \mathcal{D}\,,
\end{align}
i.e. the set of non-degenerate matrices, $\mathrm{det} \left( \mathsf{Z} \right) \neq 0$.  We prove the following:
\begin{prop}[\cite{Bykov_2024}]\label{propCotangent}
    $\mathcal{X}$ is symplectomorphic to an open subset of $\mathrm{T}^\ast \mathcal{F}_{n_0,n_1,\dots,n_r}$.
\end{prop}
\begin{proof}
    The symplectic form on $\prod_{i=0}^r \mathrm{Gr}(n_i, n)$ can be written as follows:
    \begin{align}
        \Omega = p\,\sum_{i=0}^r \Omega_{\mathrm{Gr}(n_i,n)} = i\,p\, \, \mathrm{Tr} \left(\mathrm{d} \mathsf{Z}^{\dagger} \wedge \mathrm{d} \mathsf{Z}\right)\,.\label{symplecticStructureGrassmann}
    \end{align}
    Suppose $\mathsf{Z}$ is non-degenerate, i.e. $\mathsf{Z} \in \mathcal{X}$. According to the polar decomposition theorem, there is a unique factorization
    \begin{equation}\label{ZUH}
    \mathsf{Z} = \mathsf{U}\mathsf{H}\,,
    \end{equation}
    where $\mathsf{U}$ is a unitary and $\mathsf{H}$ a positive-definite Hermitian matrix. Clearly, $\mathsf{U} = \mathsf{Z} \left(\mathsf{Z}^{\dagger} \mathsf{Z}\right)^{-1/2}$ and $\mathsf{H} = \left(\mathsf{Z}^{\dagger} \mathsf{Z}\right)^{1/2}$. For later use we introduce the matrix $\mathsf{K} := \mathsf{H}^2$. Its diagonal elements are equal to $1$ due to the normalization conditions on the vectors $z_a$. Substituting the decomposition into (\ref{symplecticStructureGrassmann}), we arrive at 
    \begin{align} \label{symplecticFormCotangent}
        \Omega\big|_{\mathcal{X}} = i\,p\, \mathrm{d}\, \mathrm{Tr}\left(\mathsf{K} \mathsf{U}^{\dagger} \mathrm{d} \mathsf{U}\right) = i \,p\, \mathrm{Tr} \left(\mathrm{d} \mathsf{U}^{\dagger} \wedge \mathrm{d} \mathsf{U}\right) + i\,p\,\mathrm{d}\,\left(\sum\limits_{j\neq k} \mathsf{K}_{jk}\, \Bar{u}_j \mathrm{d}u_k\right)\,, 
    \end{align}
    where $u_k$ denote the columns of $\mathsf{U}$. We note that one may think of $\mathsf{U}$ as parametrizing the flag manifold  $\mathcal{F}_{n_0,n_1,\dots,n_r}$ embedded into  $\prod_{i=0}^r \mathrm{Gr}(n_i, n)$. Thus, the first term in~(\ref{symplecticFormCotangent}) vanishes due to Proposition \ref{lagrtheorem}. The second term represents the standard symplectic form on $\mathrm{T}^\ast \mathcal{F}_{n_0,n_1,\dots,n_r}$. As a result,~(\ref{symplecticFormCotangent}) proves that $\mathcal{X}$ is \emph{symplectomorphic} to an open subset of $\mathrm{T}^\ast \mathcal{F}_{n_0,n_1,\dots,n_r}$ via the decomposition~(\ref{ZUH}). Recall, however, that $\mathsf{K}$ is a positive-definite matrix by construction. This condition distinguishes an open subset of the cotangent bundle.
\end{proof}

We can increase the size of the open subset in Proposition \ref{propCotangent} by changing the
value of $p \in \mathbb{Z}$ in the symplectic form~(\ref{symplecticStructureGrassmann}), which effectively controls the size of the subset.   Morally, this means that $\prod_{i=0}^r \mathrm{Gr}(n_i, n)$, with the symplectic form~(\ref{symplecticStructureGrassmann}), is a symplectic compactification of $\mathrm{T}^\ast \mathcal{F}_{n_0,n_1,\dots,n_r}$ constructed by adding the divisor $\mathcal{D}$ and taking the limit $p\mapsto\infty$.

We note that $|\mathrm{det}\left(\mathsf{Z}\right)|^2$ is a function with a (zero) minimum on $\mathcal{D}$ and a maximum on the Lagrangian submanifold $\mathcal{F}_{n_0,\dots,n_r}$ in $\prod_{i=0}^r \mathrm{Gr}(n_i, n)$ (see~\cite{Audin, Biran} for a discussion of this setup). As regards the second claim, first notice that $|\mathrm{det}\left(\mathsf{Z}\right)|^2=\mathrm{det}\left(\mathsf{Z}^\dagger\mathsf{Z}\right).$ Denote the eigenvalues of the Hermitian matrix  $\mathsf{Z}^\dagger\mathsf{Z}$ by $\sigma_1, \ldots, \sigma_n$. Due to the normalization condition $|z_a|^2=1$ on the columns of $\mathsf{Z}$ one has ${\mathrm{Tr}(\mathsf{Z}^\dagger\mathsf{Z})=\sum_{k=1}^n \sigma_k=n}$. By the inequality for the geometric mean, $\mathrm{det}\left(\mathsf{Z}^\dagger\mathsf{Z}\right)=\prod_{k=1}^n \sigma_k\leq \left({1\over n}{\sum_{k=1}^n \sigma_k }\right)^n=1$, where the equality is attained only when all eigenvalues are equal, in which case  $\mathsf{Z}^\dagger\mathsf{Z}=\mathds{1}_n$. 

Another useful fact is that the divisor $p \mathcal{D}$  is Poincar\'e dual to $[\Omega]$, the cohomology class of the symplectic form~(\ref{symplecticStructureGrassmann}).  Indeed, $\left(\mathrm{det}\left(\mathsf{Z}\right)\right)^p$ is a section of the respective line bundle $\mathcal{O}(p\mathcal{D})$~\cite{huybrechts2005complex}. From the form of the polynomial $\left(\mathrm{det}\left(\mathsf{Z}\right)\right)^p$ it follows that it is at the same time a section of the line bundle\footnote{Here by $\mathcal{L}_k$ we denote the determinant of the (anti)tautological bundle over the $k^{\mathrm{th}}$ Grassmannian, so that $c_1(\mathcal{L}_k)=[\Omega_{\mathrm{Gr}(n_k,n)}]$.} ${(\mathcal{L}_0)^p\otimes\cdots \otimes (\mathcal{L}_r)^p}$, whose first Chern class is ${c_1=p\,\sum_{i=0}^r [\Omega_{\mathrm{Gr}(n_i,n)}]=[\Omega]}$.

Finally, let us discuss the physical implications of the mathematical construction just studied. Proposition \ref{propCotangent} means that, given a Hamiltonian $\mathrm{H}$ on $\prod_{i=0}^r \mathrm{Gr}(n_i, n)$, one obtains the induced  Hamiltonian $\mathrm{H}_{\mathrm{Ind}}$ on the open subset of $\mathrm{T}^\ast \mathcal{F}_{n_0,n_1,\dots,n_r}$. For a specific choice of $\mathrm{H}$ the induced Hamiltonian $\mathrm{H}_{\mathrm{Ind}}$ is the  Hamiltonian of the geodesic flow on $\mathcal{F}_{n_0,n_1,\dots,n_r}$. Before generalizing this observation to the case of the magnetic geodesic flow, we will illustrate Propositions \ref{lagrtheorem} and \ref{propCotangent} using the simple example of the sphere~$\mathcal{S}^2 = \CP^1$.

\subsection{Geodesic flow on $\CP^1$.}

First, let us construct, in very explicit terms, the map
\begin{align}\label{CP1map}
\mathcal{X}\simeq \CP^1 \times \CP^1 \,\Bigr\backslash\, \mathcal{D} \mapsto \mathrm{T}^\ast \CP^1 \,.    
\end{align}
The two spheres in the l.h.s. are parametrized by the unit vectors $\vec{n}_1, \vec{n}_2\in \mathbb{R}^3$. The divisor $\mathcal{D}$ is defined in this case by
$
    \mathcal{D}=\Bigl\{\vec{n}_2=\vec{n_1}\Bigr\}
$. 
We also recall the definition
\begin{align}
\mathrm{T}^\ast \CP^1=\Bigl\{\;\vec{n}^2=1, \quad \vec{\pi}\cdot \vec{n}=0\;\Bigr\}  \subset \mathbb{R}^6\,.
\end{align}
In these variables the map~(\ref{CP1map}) is defined as follows: 
\begin{align}\label{1}
\vec{n}=\frac{\vec{n}_1-\vec{n}_2}{\sqrt{2(1-\vec{n}_1\vec{n}_2)}}  \,,\quad\quad 
\vec{\pi}=p\,\frac{\vec{n}_1\times \vec{n}_2}{\sqrt{2(1-\vec{n}_1\vec{n}_2)}}\,.
\end{align}

One can explicitly check that this is a symplectomorphism, which coincides with the one of  Proposition~\ref{propCotangent} above. The Hamiltonian $\mathrm{H}=\vec{\pi}^{\,2}$ describes geodesic flow\footnote{This is most easily seen by writing the corresponding Lagrangian $
    L=\vec{\pi}\,\dot{\vec{n}}-\vec{\pi}^{\,2}+\uplambda\,\vec{\pi}\cdot \vec{n}\,,
$ where $\uplambda$ is a Lagrange multiplier and, additionally,  $\vec{n}^2=1$. Successively eliminating $\vec{\pi}$ and $\uplambda$, we arrive at the sigma model (i.e., geodesic flow) Lagrangian
$
    L={1\over 4} \,\dot{\vec{n}}^2
$. 
} on~$\CP^1$. Under the above map it is pulled back exactly to the  Hamiltonian
\begin{align}
\mathrm{H}=\vec{\pi}^{\,2}=p^2\,\frac{1+\vec{n}_1\vec{n}_2}{2}\,,
\end{align}
known in physics literature as the classical counterpart of the `spin chain' Hamiltonian~\cite{Bykov_2024} (i.e. the one describing the spin-spin interaction). 

Let us also see explicitly how the solutions of the two systems are mapped into each other. The e.o.m. of the `spin' model read
\begin{align}
    \dot{\vec{n}}_1=p\,\vec{n}_1\times \vec{n}_2\,,\quad\quad
    \dot{\vec{n}}_2=p\,\vec{n}_2\times \vec{n}_1\,.
\end{align}
It follows that
$
    \vec{n}_1+\vec{n}_2=\vec{a}
$ 
is an integral of motion. Notice that, as a result,
\begin{align}\label{arestr}
|\vec{a}|\leq 2 \,.   
\end{align}
If we define $\vec{n}_1-\vec{n}_2=\vec{m}$, the remaining equation may be cast in the form $
    \dot{\vec{m}}=p\,\vec{m}\times \vec{a}
$. 
If $\vec{e}_1, \vec{e}_2$ constitute an orthonormal basis orthogonal to $\vec{a}$, the solution is
\begin{align}
    \vec{m}=A \cos{(p|a|t)} \vec{e}_1+A \sin{(p|a|t)} \vec{e}_2\,.
\end{align}
Substituting in~(\ref{1}), we get
\begin{align}
    &\vec{n}=\cos{(p|a|t)} \vec{e}_1+\sin{(p|a|t)} \vec{e}_2\,,\\
    &\vec{\pi}={p|a|\over 2}\Bigr(-\sin{(p|a|t)} \vec{e}_1+ \cos{(p|a|t)} \vec{e}_2\Bigr)\,,
\end{align}
which are the geodesic equations on $\CP^1$. Note, however, that due to~(\ref{arestr})  we get only the geodesics with restricted momenta $|\vec{\pi}|\leq p$ (which defines the open subset in the cotangent bundle $\mathrm{T}^\ast \CP^1$). The restriction is lifted in the limit $p\to \infty$.

\section{$\mathrm{SO}, \mathrm{Sp}$ cases: examples}
 In this section, we will apply the above scheme to the cases when the group $\mathrm{G}$ is  $\mathrm{SO}(2n)$, $\mathrm{SO}(2n+1)$ or $\mathrm{Sp}(n)$.

We start with the case  $\mathrm{G}=\mathrm{SO}(2n)$ and fix the Cartan subalgebra~$\mathfrak{t}$. It consists of matrices of the form
\begin{align}\label{LambdaSO}
&\Uplambda = 
\begin{pmatrix}
    0 \cdot \mathds{1}_{n_0}  & & \\
    & \lambda_1 \cdot \mathds{1}_{n_1} & \\
    & & \ddots \\ 
    & & &  \lambda_r \cdot \mathds{1}_{n_r}
\end{pmatrix}
\otimes \mathcal{J}_2\,,\\
&\mathcal{J}_2 := 
\begin{pmatrix}
    0 & 1 \\
    -1 & 0
\end{pmatrix}\,,
\end{align}
where $r \geq 0 $, $\{\lambda_i\}_{i=1}^{r} \in \mathbb{R} \backslash \{0\}$ and $n_0 + n_1 + \dots +n_r = n$. Here we have assumed that $\Uplambda$ has zero eigenvalues: if not, one can simply set $n_0=0$. Besides, we assume that all $|\lambda_i|$'s are distinct although later we will also consider the cases where some of the  $|\lambda_i|$'s are equal (but $\lambda_i$'s distinct). 

The orbit of $\Uplambda$ in~(\ref{LambdaSO}) is the \emph{generalized} flag manifold\footnote{We formally set $\mathrm{SO}(0)$ to be the trivial group.}
\begin{align}
    \mathcal{O}_{\Uplambda}^{\,\mathrm{SO}(2n)} = \frac{\mathrm{SO}(2n)}{\mathrm{SO}(2n_0)\times \mathrm{U}(n_1) \times\dots\times \mathrm{U}(n_r)}\,.
\end{align}
This has a geometric interpretation very similar to the one of~(\ref{flagofspaces}). Again, one picks an ordering of $n_1, \ldots, n_r$ (say, the lexicographic one) and defines the flag
\begin{align}
    0\subset \tilde{L}_1 \subset \cdots \subset \tilde{L}_r\,,\quad\quad \textrm{where}\quad\quad \mathrm{dim}_{\mathbb{C}} \,\tilde{L}_k=\sum_{i=1}^k\,n_i\,,
\end{align}
but in this case $\tilde{L}_k$ are linear subspaces in $\CC^n$ isotropic w.r.t. a non-degenerate symmetric form (the one preserved by the action of $\tSO(2n)$).

The case $r=1$ corresponds to the Grassmannian of isotropic $n_1$-planes in $\CC^{2n}$. Its geometry was discussed in detail in~\cite{BK}. Henceforth we will use the following notation for such Grassmannians:
\begin{align}
    \mathsf{OGr}_{m}:=\frac{\tSO(2n)}{\tSO(2n-2m)\times\tU(m)}\,.
\end{align}

Our next goal is to construct isotropic embeddings of $\mathcal{O}_{\Uplambda}^{\,\mathrm{SO}(2n)}$ following the method of Section~\ref{Embeddings}. We choose the $\Uplambda_i$'s as follows:
\begin{align}
    &\Uplambda_0 = 
    \begin{pmatrix}
    0 \cdot \mathds{1}_{n_0}  &  \\
     & -1 \cdot \mathds{1}_{n - n_0} \\
    \end{pmatrix}
    \otimes 
    \mathcal{J}_2\,,\\
    &\Uplambda_i = 
    \begin{pmatrix}
    0 \cdot \mathds{1}_{n_0+\dots+n_{i-1}}  & & \\
     & 1 \cdot \mathds{1}_{n_i} & \\
     & & 0 \cdot \mathds{1}_{n_{i+1}+\dots+n_{r}}
    \end{pmatrix}
    \otimes 
    \mathcal{J}_2\,,\;\quad \text{where}\;i = 1,2,\dots,r\,. 
\end{align}
Clearly, $\sum_{i=0}^r \Uplambda_i = 0, \;\Uplambda = \Uplambda_0 + \sum_{i=1}^r (\lambda_i + 1) \Uplambda_i$ and $\mathrm{Stab}_{\Uplambda} = \bigcap_{i=0}^r \mathrm{Stab}_{\Uplambda_i}$. Thus, we can use Proposition \ref{embeddingProp} to construct the isotropic embedding 
\begin{align}\label{SO2nEmbed1}
\mathcal{O}_{\Uplambda}^{\,\mathrm{SO}(2n)} \hookrightarrow \mathsf{OGr}_{n-n_0}\times \prod_{i=1}^r\,\mathsf{OGr}_{n_i}\,.
\end{align}
Simple dimensional analysis shows that the only possible Lagrangian embeddings of this type belong to the case where $\mathcal{O}_{\Uplambda}^{\,\mathrm{SO}(2n)}$ is a Grassmannian, which is the case of Corollary \ref{simpleLagrExample}.

Consider the special case  $n_0 = 0$. The corresponding generalized flag manifold is
\begin{align}
    {\tilde{\mathcal{O}}}_{\Uplambda}^{\,\mathrm{SO}(2n)} = \frac{\mathrm{SO}(2n)}{\mathrm{U}(n_1)\times \mathrm{U}(n_2) \times\dots\times \mathrm{U}(n_r)}\,.
\end{align}
As above, we have the isotropic embedding 
\begin{align}\label{SO2nEmbed2}
{\tilde{\mathcal{O}}}_{\Uplambda}^{\,\mathrm{SO}(2n)}\hookrightarrow \mathsf{OGr}_n\times \prod_{i=1}^r\,\mathsf{OGr}_{n_i}\,.
\end{align}
There is at least one other embedding defined by
\begin{align}
    \Uplambda_i = 
    \begin{pmatrix}
    -\frac{1}{r-1} \cdot \mathds{1}_{n_1+\dots+n_{i-1}}  & & \\
     & 1 \cdot \mathds{1}_{n_i} & \\
     & & -\frac{1}{r-1} \cdot \mathds{1}_{n_{i+1}+\dots+n_{r}}
    \end{pmatrix}
    \otimes 
    \mathcal{J}_2\,,
\end{align}
where $i = 1,\dots,r$. At the same time, we should require $\sum_{i=1}^r \lambda_i = 0$. In this case one can express $\Uplambda = \sum_{i=1}^r \mathrm{C}_i \Uplambda_i$ for some coefficients $\mathrm{C}_i$. The resulting isotropic embedding is 
\begin{align}
{\tilde{\mathcal{O}}}_{\Uplambda}^{\,\mathrm{SO}(2n)} \hookrightarrow \prod_{i=1}^r\frac{\mathrm{SO}(2n)}{\mathrm{U}\left(n-n_i\right)\times \mathrm{U}(n_i)}\,. \label{SO2nEmbed3}
\end{align}
The embedding is Lagrangian only in the case $r=2$ that leads one  back to  Corollary~\ref{simpleLagrExample}. This example demonstrates the flexibility of our construction.

\subsection{$\mathrm{SO}(2n+1)$ and $\mathrm{Sp}(n)$.} 
The cases of $\mathrm{SO}(2n+1)$ and $\mathrm{Sp}(n)$\footnote{$\mathrm{Sp}(n)$ is a compact form of $\mathrm{Sp}(2n, \mathbb{C})$. For example, $\mathrm{Sp}(1) \simeq \mathrm{SU}(2)$.} can be treated fully analogously. We first fix the Cartan subalgebras $\mathfrak{t}$: 
\begin{align}
    &\Uplambda^{\mathrm{SO}(2n+1)} = 
    \begin{pmatrix}
        \Lambda
        \otimes 
        \mathcal{J}_2 & \\
        & 0
    \end{pmatrix}\,,\quad\quad 
\Uplambda^{\mathrm{Sp}(n)} = 
    \begin{pmatrix}
        i\Lambda & \\
         & -i\Lambda
    \end{pmatrix}\,,
    \end{align}
    \begin{align}
    &\text{where}\quad\Lambda = 
    \begin{pmatrix}
        0 \cdot \mathds{1}_{n_0}  & & \\
        & \lambda_1 \cdot \mathds{1}_{n_1} & \\
        & & \ddots \\ 
        & & &  \lambda_r \cdot \mathds{1}_{n_r}
    \end{pmatrix}.
\end{align}
We assume that $|\lambda_i|$'s are real distinct numbers and $n_0 + n_1 +\dots n_r = n$. We then obtain the following embeddings: 
\begin{align}\label{SO 2n+1 flag}
&\mathcal{O}_{\Uplambda}^{\,\mathrm{SO}(2n+1)}:=\frac{\mathrm{SO}(2n+1)}{\mathrm{SO}(2n_0+1)\times \mathrm{U}(n_1) \times\dots\times \mathrm{U}(n_r)} \hookrightarrow \mathsf{OGr}^{'}_{n-n_0} \times \prod_{i=1}^r \mathsf{OGr}^{'}_{n_i}\,,\\ \label{Sp 2n flag}
    &\mathcal{O}_{\Uplambda}^{\,\mathrm{Sp}(n)}:=\frac{\mathrm{Sp}(n)}{\mathrm{Sp}(n_0)\times \mathrm{U}(n_1) \times\dots\times \mathrm{U}(n_r)} \hookrightarrow \mathsf{SGr}_{n-n_0} \times \prod_{i=1}^r \mathsf{SGr}_{n_i}\,,\\
    &\text{where}\quad \mathsf{OGr}^{'}_{m}:=\frac{\mathrm{SO}(2n+1)}{\mathrm{SO}(2n-2m+1) \times \mathrm{U}(m)}\,,\quad \mathsf{SGr}_{m}:=\frac{\mathrm{Sp}(n)}{\mathrm{Sp}(n-m) \times \mathrm{U}(m)}\,.
\end{align}
It is easy to see that these embeddings have the same form as (\ref{SO2nEmbed1}) up to a suitable change from $\mathrm{SO}(2k)$ to $\mathrm{SO}(2k+1)$ or $\mathrm{Sp}(k)$. The same holds true when $n_0 = 0$. There are embeddings of the type~(\ref{SO2nEmbed2}) and~(\ref{SO2nEmbed3}), but there are no Lagrangian embeddings except for the case described in Corollary \ref{simpleLagrExample}.

Finally, we note that, just like in the case of the groups $\mathrm{SU}(n)$ and $\tSO(2n)$, the orbits in the l.h.s. of~(\ref{SO 2n+1 flag})-(\ref{Sp 2n flag}) are the manifolds of flags of linear subspaces isotropic w.r.t. a non-degenerate quadratic form: this form is symmetric in the case of $\tSO(2n+1)$ and skew-symmetric in the case of $\mathrm{Sp}(n)$.

\begin{examp}
A Lagrangian example.    
\end{examp}

The expectation is that, for any Grassmannian $X_\mathrm{Gr}$, the only nontrivial embedding is into its square, i.e. $X_\mathrm{Gr} \hookrightarrow X_\mathrm{Gr}\times X_\mathrm{Gr}$. Thus to arrive at some less trivial examples we should look beyond the realm of Grassmannians.

In the case of $\mathrm{SO}(4)$ the only two orbits are ${\mathrm{SO}(4) \over \mathrm{U}(1)^2}\simeq \CP^1\times \CP^1$ and ${\mathrm{SO}(4) \over \mathrm{U}(2)}\simeq \CP^1$, which are Grassmannians. In the case of $\mathrm{SO}(6)$ we have the following orbits:
\newcolumntype{C}{>{$}c<{$}}
\begin{center}
   \begin{tabular}{C|C|C|C|C}
        \text{Orbit} & {\mathrm{SO}(6) \over \mathrm{U}(1)^3} & {\mathrm{SO}(6) \over \mathrm{U}(1)\times \mathrm{U}(2)} & {\mathrm{SO}(6) \over \mathrm{U}(1)\times \mathrm{SO}(4)} &{\mathrm{SO}(6) \over \mathrm{U}(3)}  \\ \text{Dimension} & 
        12 & 10 & 8 & 6
    \end{tabular} 
\end{center}
Only the first one of these is non-Grassmannian, and in this case one can construct {the~embedding}
\begin{align}
    {\mathrm{SO}(6) \over \mathrm{U}(1)^3} \hookrightarrow \left({\mathrm{SO}(6) \over \mathrm{U}(3)}\right)^{\times 4}\,.
\end{align}
The corresponding orbits are defined by the matrices 
\begin{align}
    \Uplambda_1 =\begin{pmatrix}
        1 & 0 & 0\\
        0 & 1 & 0\\
        0 & 0 & 1
    \end{pmatrix}\otimes \mathcal{J}_2\,,\quad\quad &\Uplambda_2=
    \begin{pmatrix}
        1 & 0 & 0\\
        0 & -1 & 0\\
        0 & 0 & -1
    \end{pmatrix}\otimes \mathcal{J}_2\,,\quad\quad \Uplambda_3 =\begin{pmatrix}
        -1 & 0 & 0\\
        0 & 1 & 0\\
        0 & 0 & -1
\end{pmatrix}\otimes \mathcal{J}_2\,, \nonumber \\
     &\Uplambda_4 =\begin{pmatrix}
        -1 & 0 & 0\\
        0 & -1 & 0\\
        0 & 0 & 1
    \end{pmatrix}\otimes\mathcal{J}_2\,.
\end{align}
Notice that for any $k$ one has $\Uplambda_k^2=-\mathds{1}_6$, so that $\Uplambda_k$ defines a complex structure in~$\mathbb{R}^6$. As a result, $\mathrm{Stab}(\Uplambda_k)=\mathrm{U}(3)$ (and not $\mathrm{U}(1)\times \mathrm{U}(2)$  as one might naively think in the last three cases), so that indeed each orbit is $\mathrm{SO}(6)/\mathrm{U}(3)$. One can additionally check that $\bigcap_{i=1}^4 \mathrm{Stab}(\Uplambda_i)=\mathrm{U}(1)^3$. 
Clearly $\Uplambda_1+\Uplambda_2+\Uplambda_3+\Uplambda_4=0$, so that the embedding is Lagrangian.

\subsection{
An exceptional Lagrangian series.   }

We start with the manifold $\frac{\tSO(2n)}{\tU(1)\times \tU(n-1)}$ describing flags of the type $(x, y)$, where $x$ is an isotropic line, $y$ an isotropic $(n-1)$-plane and $x\oplus y$ an isotropic $n$-plane\footnote{By $\oplus$ we denote the plane spanned by the respective vectors.}. Besides, we assume that $y$ is orthogonal to~$x$ w.r.t. the standard sesquilinear scalar product in $\CC^{2n}$. We will denote by $y^\ast$ the isotropic $(n-1)$-plane whose basis is obtained by complex conjugation of the basis of~$y$ (hence $y^\ast$ is orthogonal to  $x\oplus y$).  
\begin{prop}\label{SOFlagLagrProp}
    Consider the  embedding 
    \begin{align}\label{SO2nembed}
        &\frac{\tSO(2n)}{\tU(1)\times \tU(n-1)}\hookrightarrow \frac{\tSO(2n)}{\tU(n)}\times \frac{\tSO(2n)}{\tU(n)}\times \frac{\tSO(2n)}{\tSO(2n-2)\times \tU(1)}\,\\
        &(x,y) \mapsto \Bigl(x \oplus y, x\oplus y^\ast, x\Bigr)
    \end{align} 
    It is Lagrangian w.r.t. a suitable choice of symplectic form in the r.h.s.
\end{prop}
\begin{proof}
 The first and second factors in the r.h.s.~of (\ref{SO2nembed}) are the Grassmannians of isotropic $n$-planes in $\CC^{2n}$  whereas the third factor is the `quadric' , i.e. the variety ${w^tw=0\subset \CP^{2n-1}}$ (cf.~\cite{BK}). 
One easily calculates the dimension of l.h.s. to be $n^2+n-2$, which is half the dimension of the r.h.s. Let us prove that the l.h.s. can be described as $\mu^{-1}(0)$, where $\mu$ is the moment map for the diagonal action of $\tSO(2n)$. As a result, it is isotropic (see~(\ref{orbitisotropic})) and hence Lagrangian. 

Let us parametrize the  Grassmannian of isotropic $n$-planes using the vectors $u_1, \ldots, u_n$ satisfying $u_i^t u_j=0$ and $u_i^\dagger u_j=\delta_{ij}$ and the quadric using a single vector~$w$ ($w^tw=0, w^\dagger w=1$). The corresponding moment maps for the standard choice of symplectic forms are given by 
\begin{align}\label{muU}
&\mu_u=\sum_{i=1}^n\,\left(u_i\otimes u_i^\dagger-u_i^\ast \otimes u_i^t\right)\in \mathfrak{so}_{2n}\,,\\
&\mu_w=w\otimes w^\dagger-w^\ast \otimes w^t\,\in \mathfrak{so}_{2n}\,.
\end{align}
Assuming the first and second $n$-Grassmannians in~(\ref{SO2nembed}) are parametrized by vectors $u$ and $v$, we pick the symplectic form on the r.h.s. of~(\ref{SO2nembed}) in such a way that the full moment map equation takes the form (notice the factor of $2$ in the last term)
\begin{equation}\label{mueq}
    \mu=\mu_u+\mu_v-2 \mu_w=0\,.
\end{equation}
The key point is that each moment map, $\mu_u, \mu_v$ or $\mu_w$, is a difference of two projectors. As a result, it has eigenvalues $-1, 1$ and possibly $0$ implying that for any vector $x$, such that $x^\dagger x=1$, one has the inequality 
\begin{align}
    -1\leq x^\dagger \mu_{u}x\leq 1
\end{align}
(and similarly for all other moment maps). Taking the analogous matrix element of equation~(\ref{mueq}) and setting $x=w$, one finds $w^\dagger \mu_u w+w^\dagger \mu_v w-2=0$, which by the inequality above implies $w^\dagger \mu_u w=w^\dagger \mu_v w=1$. In turn, this means that  ${\mu_u w=\mu_v w=w}$. Looking back at the definition~(\ref{muU}) and taking scalar products with $u_i^t$ one finds that $u_i^t w=0$ and analogously $v_i^t w=0$, implying that
\begin{equation}
    w\in \mathrm{Span}\left(\Bigl\{ u_i \Bigr\}_{i=1}^n\right)\bigcap \mathrm{Span}\left(\Bigl\{ v_i \Bigr\}_{i=1}^n\right)\,.
\end{equation}
Without loss of generality, one may set $u_1=v_1=w$, in which case equation~(\ref{mueq}) turns into
\begin{align}
\mu_{u^\perp}+\mu_{v^\perp}=0\,,    
\end{align}
where the index $\perp$ refers to the projection to the subspace orthogonal to $w, w^\ast$, i.e. $\mu_u^\perp=\sum_{i=2}^n\,\left(u_i\otimes u_i^\dagger-u_i^\ast \otimes u_i^t\right)$. Clearly, the latter equation means that ${\mathrm{Span}\left(\Bigl\{ u_i \Bigr\}_{i=2}^n\right)=\mathrm{Span}\left(\Bigl\{ v_i^\ast \Bigr\}_{i=2}^n\right)}$.

It is also clear from above that the map $(w, u^\perp)\mapsto \Bigl(u, v, w\Bigr)$ is the one described in the proposition.
\end{proof}

\section{General embeddings $\mu=0$}

To start with, we prove the following lemma (we assume $\mu^{-1}(0)$ connected):

\begin{lem}
    An orbit $L\subset \mu^{-1}(0)$ is Lagrangian only if $L\simeq \mu^{-1}(0)$.
\end{lem}

\begin{proof}
    Start with the definition of the moment map: $d\mu[v]=\omega(\bullet, v)$, where $v$ is a vector field generating the Lie group action. Take a point $p\in L$: contracting the above equality with an arbitrary vector field tangent to $\mu^{-1}(0)$ 
    we find that 
    \begin{align}\label{omegaWV}
        0=\omega(w, v)\,,\quad\quad \textrm{where}\quad v\in  \mathrm{T}_pL\,,\quad\quad w\in \mathrm{T}_p\mu^{-1}(0)
    \end{align}
    Now, suppose $v_1 \ldots v_\ell$ is a basis of vectors tangent to the orbit $L\subset \mu^{-1}(0)$ but not a basis in $\mathrm{T}_p\mu^{-1}(0)$. In this case there is a non-zero vector  $w\in \mathrm{T}_p\mu^{-1}(0)/\mathrm{T}_pL$. By~(\ref{omegaWV}) the restriction of $\omega$ to $\mathrm{Span}(\mathrm{T}_pL, w)$ is zero. As a result, $L$ is not Lagrangian unless $\mathrm{T}_p\mu^{-1}(0)\simeq \mathrm{T}_pL$ at every point implying that $\mu^{-1}(0)\simeq L$. 
\end{proof}

Thus, if one wishes to find Lagrangian submanifolds within $\mu^{-1}(0)$, one should search for cases where $\mu^{-1}(0)$ is a single orbit. Let us recall how this works in the case of $\mathrm{SU}(n)$ flag manifolds.

\subsection{$\mathrm{SU}(n)$ flag manifolds.}

We will now use the method employed in the proof of Proposition~\ref{SOFlagLagrProp}  to show that the flag manifold featuring in Proposition~\ref{lagrtheorem} is $\mu^{-1}(0)$. Indeed, for any matrix $\mu_i:=g_i\Uplambda_i g_i^\dagger$ in the orbit of $\Uplambda_i$ given by~(\ref{Lambda U(n) Grass}) one has the following bound:
\begin{align}\label{Un moment map bound}
    -n_i\leq x^\dagger  \mu_i x\leq n-n_i\,,
\end{align}
where $x$ is an arbitrary  vector such that $x^\dagger x=1$. The vanishing of the moment map corresponding to the diagonal action of $\mathrm{SU}(n)$ gives
\begin{align}
    \sum\limits_{i=1}^r\,\mu_i=0\,.
\end{align}
Take an eigenvector $x$ of the $\ell$-th moment map $\mu_\ell$ with eigenvalue $n-n_\ell$, i.e. ${\mu_\ell x=(n-n_\ell) x}$ (we additionally assume $x^{\dagger}x=1$),  then 
\begin{align}
    \sum\limits_{i\neq \ell}^r\,x^\dagger \mu_ix+(n-n_{\ell})=0\,.
\end{align}
Due to the bound~(\ref{Un moment map bound}) this is only satisfied whenever $\mu_ix=-n_i x$ for all $i\neq \ell$. This means that $x$ does not lie in any of the planes corresponding to the other Grassmannians. One can thus project to the subspace orthogonal to the corresponding  eigenspace of $\mu_\ell$ and proceed inductively. One will then find that the eigenspaces corresponding to positive eigenvalues of the $\mu_i$'s  are mutually orthogonal. This proves that $\mu^{-1}(0)$ is the flag manifold in the l.h.s. of~(\ref{flagEmbedding}).

\subsection{Embeddings of generalized flag manifolds.}

The requirement that $\mu^{-1}(0)$ be a single orbit is rather stringent: typically $\mu^{-1}(0)$ contains families of orbits. 
However, even if $\mu^{-1}(0)$ is a single orbit, in general it by no means has to be Lagrangian (for an illustration see~Appendix~\ref{regularappendix}). 
In this section  we will construct a family of \emph{isotropic} embeddings of generalized flag manifolds such that  $\mu^{-1}(0)$ is a single orbit.

Consider an arbitrary orbit
\begin{align}\label{Ogenembed}
\mathcal{O}^{\,\mathrm{SO}(2n)}_{n_1, \ldots, n_r}:=\frac{\tSO(2n)}{\tU(n_1)\times \cdots \times \tU(n_r)\times \tSO(2n-2m)}\,,\quad\quad m=n_1+\ldots + n_r \leq n    \,.
\end{align}
We will assume that the integers $n_1 \ldots n_r$ are ordered lexicographically; 
besides, we set $d_{\ell}=\sum_{s=1}^\ell n_s$.

\begin{prop}
    Consider the product manifold
    \begin{align}
\mathcal{N}^{\,\mathrm{SO}(2n)}=\mathsf{OGr}_{d_1}\times \mathsf{OGr}_{d_2}\times \cdots \times \mathsf{OGr}_{m}\times \mathsf{OGr}_{m}\,,
    \end{align}
    endowed with a product symplectic form. The symplectic form on each orbit is defined using the following set of $\Uplambda$'s ($\Uplambda_\ell^{\mathrm{SO}(2n)}:=\Upsilon_{\ell}\otimes \mathcal{J}_2$):
\begin{align}\nonumber
& \Upsilon_1=2^{r-1} \,\mathrm{Diag}\Bigl( \mathds{1}_{n_1}, \mathbf{0}_{n_2}, \ldots \Bigr)\,,\quad\quad \\ \nonumber
&\Upsilon_2=2^{r-2} \,\mathrm{Diag}\Bigl( -\mathds{1}_{n_1}, \mathds{1}_{n_2}, \mathbf{0}_{n_3}, \ldots \Bigr)\,,\\ \label{SO general Upsilon}
& \Upsilon_3=2^{r-3} \,\mathrm{Diag}\Bigl( -\mathds{1}_{n_1}, -\mathds{1}_{n_2}, \mathds{1}_{n_3}, \mathbf{0}_{n_4}, \ldots \Bigr)\,,\quad  \\ \nonumber
& \hspace{1cm}\vdots \\ \nonumber
&\Upsilon_{r}= \,\mathrm{Diag}\Bigl( -\mathds{1}_{n_1}, -\mathds{1}_{n_2}, \ldots,  -\mathds{1}_{n_{r-1}}, \mathds{1}_{n_r}, \mathbf{0}_{n-m}  \Bigr)\,,\quad\\
&\Upsilon_{r+1}= \,\mathrm{Diag}\Bigl( -\mathds{1}_{n_1}, -\mathds{1}_{n_2}, \ldots,  -\mathds{1}_{n_r} , \mathbf{0}_{n-m}  \Bigr)\,.\nonumber
\end{align} 
Then 
    there is an isotropic embedding  
\begin{align}\label{embedSOgen}
    \mathcal{O}^{\,\mathrm{SO}(2n)}_{n_1, \ldots, n_r} \hookrightarrow 
    \mathcal{N}^{\,\mathrm{SO}(2n)}
\end{align}
such that $\mathcal{O}^{\,\mathrm{SO}(2n)}_{n_1, \ldots, n_r}\simeq \mu^{-1}(0)$, where~$\mu$ is the  moment map corresponding to the diagonal action of $\mathrm{SO}(2n)$.  The map~(\ref{embedSOgen}) is given by
\begin{align}
    (x_1, x_2, \ldots, x_r) \mapsto \Bigl(x_1, x_1^\ast\oplus x_2, \ldots, x_1^\ast\oplus \cdots \oplus x_{r-1}^\ast\oplus x_{r}, x_1^\ast\oplus \cdots \oplus x_{r-1}^\ast\oplus x_r^\ast\Bigr)\,,
\end{align}
where  $(x_1, x_2, \ldots, x_r)$ is  the $r$-tuple of mutually orthogonal linear spaces laying inside the isotropic plane $x_1\oplus \cdots \oplus x_r$.
\end{prop}

\begin{proof}
The relevant moment map is 
\begin{align}\label{O(2n) full moment map}
\mu=i\,\sum\limits_{\ell=1}^{r+1}\,g_\ell \Uplambda_\ell^{\mathrm{SO}(2n)} g_\ell^t\,,
\end{align}
where $g_\ell \in \tSO(2n)$ for every $\ell$. Besides, one easily checks  that 
\begin{align}\label{power2sum}
\left(\sum_{\ell=1}^{r+1}\,\Upsilon_{\ell}\right)_{ii}=2^{r-i}-\sum_{j=i+1}^{r} 2^{r-j}-1=0
\end{align}
so that $\sum_{\ell=1}^{r+1}\,\Uplambda_{\ell}^{\mathrm{SO}(2n)}=0$. Thus, by Proposition~\ref{embeddingProp}, the orbit $\mathcal{O}^{\,\mathrm{SO}(2n)}_{n_1, \ldots, n_r}\subset \mu^{-1}(0)$  is  isotropic. 

Next we wish to show that, in fact, $\mathcal{O}^{\,\mathrm{SO}(2n)}_{n_1, \ldots, n_r}\simeq \mu^{-1}(0)$. To this end, notice that the matrices  $\mu_{\ell}:=i\,g_{\ell} \Uplambda_\ell^{\mathrm{SO}(2n)} g_{\ell}^t$ with $\ell=1, \ldots, r$ have eigenvalues $0, \pm 2^{r-\ell}$; the matrix $\mu_{r+1}$ has eigenvalues~$0, \pm 1$, so that for any $x$ with $x^\dagger x=1$ one has the bounds
\begin{align}\label{mubound1}
    &-2^{r-\ell}\leq x^\dagger \mu_{\ell} x \leq 2^{r-\ell}\,,\quad\quad \ell=1, \ldots, r\\ 
    \label{mubound2}&-1\leq x^\dagger \mu_{r+1}x \leq 1
\end{align}
Next, take for $x$ any eigenvector of $\mu_1$  with eigenvalue $2^{r-1}$:
\begin{align}
    \mu_1 x=2^{r-1} x
\end{align}
Then from $\mu=0$, using the explicit expression~(\ref{O(2n) full moment map}), we get $2^{r-1}+\sum\limits_{\ell=2}^{r+1}\,x^\dagger \mu_{\ell} x=0$. Due to the bounds~(\ref{mubound1})-(\ref{mubound2}) and the property~(\ref{power2sum}) this is only possible if ${x^\dagger \mu_{\ell} x=-2^{r-\ell}}$ for $\ell=2, \ldots, r$ and $x^\dagger \mu_{r+1} x=-1$. It then follows that $x$ is an eigenvector of all $\mu_\ell$'s:
\begin{align}
    &\mu_\ell x=-2^{r-\ell} x\,,\quad \ell=2, \ldots, r\\
    &\mu_{r+1} x=-x
\end{align}
This means that if $(y_1, y_2, \ldots, y_r, y_{r+1})$ is a point in the product of Grassmannians in the r.h.s. of~(\ref{embedSOgen}) such that $\mu=0$ at this point, then $y_1^\ast\subset y_2, y_3$, etc. We may thus express the remaining $y_i$'s as follows: $y_\ell=y_1^\ast \oplus y_{\ell}^\perp$ for $\ell=2, \ldots r+1$. Here $y_{\ell}^\perp$ comprises those vectors from $y_{\ell}$ that are orthogonal to $y_1, y_1^\ast$.

It is also easy to write out the truncated system of equations on $y_\ell^\perp$, which follow from~$\mu=0$. These equations are the projection of $\mu=0$ to the subspace orthogonal to $y_1, y_1^\ast$ of dimension $2n-2n_1$. It has the form
\begin{align}
\mu_1^\perp=i\,\sum\limits_{\ell=2}^{r+1}\,\tilde{g}_\ell \tilde{\Uplambda}_\ell^{\mathrm{SO}(2n)}\tilde{g}_\ell^t=0\,.
\end{align}
Here $\tilde{g}_\ell\in\tSO(2n-2n_1)$ and $\tilde{\Uplambda}_\ell^{\mathrm{SO}(2n)}:=\tilde{\Upsilon}_{\ell}\otimes \mathcal{J}_2$, where $\tilde{\Upsilon}_{\ell}$ ($\ell=2, \ldots, r+1$) are obtained from $\Upsilon_\ell$ by truncating the first entry of each matrix in~(\ref{SO general Upsilon}). Thus we may inductively repeat the argument to find that $(y_2^\perp)^\ast \subset y_3, y_4$, etc. At the end of the induction we arrive at the claim that $\mu^{-1}(0)$ is the flag manifold in the l.h.s. of~(\ref{embedSOgen}). 
\end{proof}

This result is easily extended to the case of $\mathrm{SO}(2n+1)$ and $\mathrm{Sp}(n)$ generalized flag manifolds~(\ref{SO 2n+1 flag}), (\ref{Sp 2n flag}). To formulate this extension, define the corresponding products of orbits
\begin{align}
    &\mathcal{N}^{\,\mathrm{SO}(2n+1)}=\mathsf{OGr}^{'}_{d_1}\times \mathsf{OGr}_{d_2}^{'}\times \cdots \times \mathsf{OGr}_{m}^{'}\times \mathsf{OGr}_{m}^{'}\,,\\
    &\mathcal{N}^{\,\mathrm{Sp}(n)}=\mathsf{SGr}_{d_1}\times \mathsf{SGr}_{d_2}\times \cdots \times \mathsf{SGr}_{m}\times \mathsf{SGr}_{m}\,,
\end{align}
where we take the Grassmannians of the respective groups.

\begin{prop}
    One has the isotropic embeddings
    \begin{align}
        &\mathcal{O}_{\Uplambda}^{\mathrm{SO}(2n+1)} \hookrightarrow \mathcal{N}^{\,\mathrm{SO}(2n+1)}\,,\\
        &\mathcal{O}_{\Uplambda}^{\mathrm{Sp}(n)} \hookrightarrow \mathcal{N}^{\,\mathrm{Sp}(n)}\,,
    \end{align}
    where the sets of $\Uplambda$'s defining the symplectic forms on $\mathcal{N}^{\,\mathrm{SO}(2n+1)}$ and $\mathcal{N}^{\,\mathrm{Sp}(n)}$ are as follows:
    \begin{align}
        &\Uplambda_\ell^{\mathrm{SO}(2n+1)}=\mathrm{Diag}(\Uplambda_\ell^{\mathrm{SO}(2n)}, 0)\,,\quad\quad \ell=1, \ldots, r+1\\
        & \Uplambda_\ell^{\mathrm{Sp}(n)}=\Upsilon_\ell \otimes \sigma_3\,,\quad\quad \ell=1, \ldots, r+1
    \end{align}
    where $\sigma_3=\begin{pmatrix}
        1& 0\\ 0& -1
    \end{pmatrix}$ and $\Upsilon_\ell$ are the matrices~(\ref{SO general Upsilon}). In both cases the embedded submanifolds are the level sets $\mu=0$ of the moment maps for the diagonal action of the symmetry groups. 
\end{prop}

\begin{proof}
   In the $\mathrm{SO}(2n+1)$ the proof is exactly as before.
   
   In the symplectic case the only difference is in the explicit form of the moment map for each Grassmannian $\mathsf{SGr}_k$, which reads (for comparison with $\mathrm{SO}(2n)$ see~(\ref{muU}))
   \begin{align}
       \mu_u^{\mathrm{Sp}(n)}=\sum\limits_{j=1}^k\,\left(u_j\otimes u_j^\dagger+\omega\, u_j^\ast \otimes u_j^t \,\omega\right)\in \mathfrak{sp}_n\,,
   \end{align}
   where $\omega=\mathds{1}_n\otimes \mathcal{J}_2$ is the constant symplectic form\footnote{We only use that $\omega$ is real and satisfies $\omega^2=-\mathds{1}_{2n}$.} on $\CC^{2n}$. Here the vectors $u_j$ parametrize an isotropic $k$-plane, i.e. $u_i^t \omega u_j=0$, and are unit-normalized: $u_i^\dagger  u_j= \delta_{ij}$. It is easy to see that $\{ u_i\}$ and $\{\omega u_i^\ast\}$ are eigenvectors of $\mu_u^{\mathrm{Sp}(n)}$ with eigenvalues $\pm 1$.  The rest of the proof goes through, as before. 
\end{proof}

\subsection{Quantization.} Although in the present paper we focus on  classical theory, let us briefly pause to outline the meaning of  condition $\mu=0$ in the quantum case.

To start with, the moment map for the diagonal action of a group $\mathrm{G}$ on a product of orbits is clearly a sum of the respective moment maps on each orbit: $\mu=\sum_k\,\mu_k$. Let us additionally assume that the symplectic form $\Omega_k$ on each orbit has the property that $[\Omega_k]=c_1(\mathcal{L}_k)$, where $\mathcal{L}_k$ is a very ample holomorphic line bundle over this orbit. In this case the lore of geometric quantization says that the Hilbert space that arises upon quantizing such orbit is $W_k:=H^0(\mathcal{L}_k)$, the space of holomorphic sections of $\mathcal{L}_k$, which is at the same time an irreducible representation of $\mathrm{G}$. Denote by $\rho_k: \mathfrak{g}\mapsto \mathrm{End}(W_k)$ the corresponding representation of Lie algebras. In this case, upon quantization the moment map $\mu_k[a]$ corresponding to a Lie algebra element $a\in \mathfrak{g}$ maps to $\rho_k(a)$, the image of $a$ in the representation $W_k$. 

To summarize, upon quantization we replace the product of orbits $\prod_i\,\mathcal{O}_{\Uplambda_i}^{\mathrm{G}}$ by the tensor product of representations $\mathcal{W}:=\otimes_k W_k$, and the constraint $\mu=0$ may now be interpreted as a constraint specifying a subspace $\mathcal{W}_0\subset \mathcal{W}$:
\begin{align}
    \mathcal{W}_0:= \Bigl\{ \quad |w\rangle \in \mathcal{W}:\quad\quad \sum_k\,\rho_k(a)\,|w\rangle=0\quad\quad \text{for any}\quad a\in\mathfrak{g}\quad  \Bigr\}
\end{align}
In other words, $\mathcal{W}_0$ is the space of trivial representations (singlets) of $\mathfrak{g}$ inside $\mathcal{W}$. For example, the condition that $\mu^{-1}(0)$ be non-empty is a classical counterpart of the existence of at least one trivial representation in $\mathcal{W}$. The condition that $\mu^{-1}(0)$ be a single orbit should imply that there is exactly one trivial representation in $\mathcal{W}$. More generally, the symplectic quotient $\mu^{-1}(0)/\mathrm{G}$ describes the classical geometry of the space of such singlets, providing an effective description thereof in the semi-classical limit $p\to \infty$.

We leave a detailed analysis of the quantum theory for the future (cf.~\cite{Bykov_2024, BKK} for the $\mathrm{SU}(n)$ case though), here we just wish to stress that this setup has been thoroughly studied from a different perspective in~\cite{Klyachko, Fulton, Knutson}.

\section{Magnetic geodesics}
We will present a practical application of the above theory for the $\SU(n)$ group by finding magnetic geodesics on the respective flag manifolds.

Let us start by amplifying Proposition \ref{propCotangent}. What would change if we consider a generic symplectic form on $\prod_{i=0}^r \mathrm{Gr}(n_i, n)$? The latter has the form
\begin{gather}
    \Omega_{\text{M}} = i\,\mathrm{Tr}\left(\mathrm{d}\mathsf{Z}^{\dagger}\wedge\mathrm{d}\mathsf{Z}\right),
\end{gather}
where $\mathsf{Z} = \begin{pmatrix}
    \mathsf{Z}_0 & \mathsf{Z}_1 & \dots & \mathsf{Z}_r
\end{pmatrix}$ and $\mathsf{Z}^{\dagger}_i \mathsf{Z}_i = p_i \times \mathds{1}_{n_i}$. Now, we will follow the same procedure as in the proof of Proposition \ref{propCotangent}, that is, for generic $\mathsf{Z} \in \mathcal{X} = \{\,\mathrm{det}\left( \mathsf{Z} \right) \neq 0\,\}$, we pass to the polar decomposition $\mathsf{Z} = \mathsf{U} \mathsf{H}$. Defining $\mathsf{K} := \mathsf{H}^2$, we get
\begin{align}
    \Omega_{\text{M}}\big|_{\mathcal{X}} = i \, \mathrm{Tr} \left(\mathsf{P}\,\mathrm{d} \mathsf{U}^{\dagger} \wedge \mathrm{d} \mathsf{U}\right) + i\,\mathrm{d}\,\left(\sum\limits_{j\neq k} \mathsf{K}_{jk}\, \Bar{u}_j \mathrm{d}u_k\right)\equiv \Omega_{\text{M}}^0\,, \label{OmegaMagnetic}
    \\
    \text{where}\quad \mathsf{P} = 
    \begin{pmatrix}
        p_0 \mathds{1}_{n_0} & & &\\
        & p_1 \mathds{1}_{n_1} & &\\
        & & \ddots \\
        & & & p_r \mathds{1}_{n_r} \\
    \end{pmatrix}\,.
\end{align}
We should note that the shift $p_i \, \mapsto \,p_i + 1$ does not change $\Omega_{\mathrm{M}}\big|_{\mathcal{X}}$, because $\mathrm{Tr} \left(\mathrm{d} \mathsf{U}^{\dagger} \wedge \mathrm{d} \mathsf{U}\right) = 0$, which follows directly from Proposition \ref{lagrtheorem}. Without loss of generality we can assume $\mathrm{Tr}\left(\mathsf{P}\right)=0$, so that the first term in (\ref{OmegaMagnetic}) is a Kirillov-Kostant-Souriau symplectic form on the flag manifold $\mathcal{F}_{n_0,\dots,n_r}$. The entire expression for $\Omega_{\mathrm{M}}\big|_{\mathcal{X}}$ is a standard `twisted' symplectic form $\Omega_{\mathrm{M}}^{0}$ for the magnetic geodesic flow on $\mathrm{T}^\ast \mathcal{F}_{n_0,n_1,\dots,n_r}$~\cite{Efimov2005}. Therefore, we have proved the following: 
\begin{prop}\label{magneticProp}
    $\mathcal{X} \subset \Big(\prod_{i=0}^r \mathrm{Gr}(n_i, n), \Omega_{\mathrm{M}}\Big)$ is symplectomorphic to an open subset of $\Big(\mathrm{T}^\ast \mathcal{F}_{n_0,\dots,n_r},\Omega_{\mathrm{M}}^{0}\Big)$.
\end{prop}
The numbers $q_i := p_i - p_{i-1}$ for $i=1,\dots,r$ turn out to be magnetic charges. As in the non-magnetic case, the size of the subset of $\mathrm{T}^\ast \mathcal{F}_{n_0,\dots,n_r}$ is determined by the value of $\min\left(p_i\right)$. 

Now, we turn to the study of magnetic geodesics. Proposition \ref{magneticProp} suggests that we can try to find a Hamiltonian $\mathrm{H}$ on $\prod_{i=0}^r \mathrm{Gr}(n_i, n)$, which induces the Hamiltonian of geodesic flow on $\mathrm{T}^\ast \mathcal{F}_{n_0,\dots,n_r}$. Then, studying its dynamics on $\prod_{i=0}^r \mathrm{Gr}(n_i, n)$, we can find magnetic geodesics\footnote{With restricted momenta due to the constraint that $\mathsf{K}$ be positive-definite.} on $\mathcal{F}_{n_0,\dots,n_r}$. For simplicity, we will restrict to the case of the complete flag manifold
\begin{align}
    \mathcal{F}_{n}:=\mathcal{F}_{1,1,\dots,1}=\frac{\mathrm{SU}(n)}{\mathrm{S}\left(\mathrm{U}(1)^{\times n}\right)}\,,
\end{align}
however the same technique can be applied to the general case of $\mathcal{F}_{n_0,\dots,n_r}$. 

We start with the following action for a dynamical system on ${\prod_{i=0}^{n-1} \mathrm{Gr}(1, n) = \left(\CP^{n-1}\right)^{\times n}}$:
\begin{gather}
    \mathcal{S} =  \int^{\tau}_{0}\,\mathrm{d}t\,\left(i\,\mathrm{Tr}\left(\mathsf{Z}^{\dagger} \dot{\mathsf{Z}}\right) - \sum_{i<j} \alpha_{ij} |\Bar{z}_i z_j|^2\right)\,, \label{ActionSpin}
\end{gather}
where $z_i$'s are columns of the matrix $\mathsf{Z}$ satisfying the normalization conditions $\Bar{z}_i z_i := p_i$. As it follows from (\ref{ActionSpin}) $\sum_{i<j} \alpha_{ij} |\Bar{z}_iz_j|^2$ plays role of Hamiltonian of the system. In physics literature the above system is referred to as a `classical spin chain' \cite{Bykov_2024}. 

To proceed, we again use the polar decomposition $\mathsf{Z} = \mathsf{U}\mathsf{H}$ for generic $\mathsf{Z}$ to cast the action in the form\footnote{The term $\mathrm{Tr}(\dot{\mathsf{K}})=0$, due to normalization conditions on $z_i$'s.}
\begin{align}
    \mathcal{S} =  \int^{\tau}_{0}\,\mathrm{d}t\,\left(i\,\mathrm{Tr}\left(\mathsf{P}\mathsf{U}^{\dagger} \dot{\mathsf{U}}\right) + \sum\limits_{j\neq k} \mathsf{K}_{jk}\, \Bar{u}_j \dot{u}_k - \sum_{i<j} \alpha_{ij} \mathsf{K}_{ij}\mathsf{K}_{ji}\right)\,.
\end{align}
From the last expression it is clear that $\mathsf{K}_{ij}$'s are momenta of the system and the Hamiltonian $\sum_{i<j} \alpha_{ij} \mathsf{K}_{ij}\mathsf{K}_{ji}$ is simply the geodesic Hamiltonian. Indeed, solving classical equations of motion on $\mathsf{K}_{ij}$'s, we end up with
\begin{align}\label{finalAction}
    \mathcal{S} =  \int^{\tau}_{0}\,\mathrm{d}t\,\left(i\,\mathrm{Tr}\left(\mathsf{P}\mathsf{U}^{\dagger} \dot{\mathsf{U}}\right) + \sum_{i<j} \frac{1}{\alpha_{ij}} \Big(\dot{\Bar{u}}_i u_j\Big) \Big(\Bar{u}_j \dot{u}_i\Big)\right)\,.
\end{align}
This is an action for magnetic geodesics on $\mathcal{F}_{n}$, equipped with the metric\footnote{Metrics on flag manifolds have been studied in~\cite{AlePer86, Arvanitoyeorgos_1993}, see also~\cite{AffleckReview}.}
\begin{align}
    \mathrm{d}s^2 = \sum_{i<j} \frac{1}{\alpha_{ij}} \Big(\mathrm{d}\Bar{u}_i u_j \Big) \Big( \Bar{u}_j \mathrm{d}u_i \Big).    
\end{align} 
Recall however that $\mathsf{K}_{ij}$'s are not arbitrary due to the positive-definiteness of $\mathsf{K}$ and the derivation works in a some open subset of $\mathrm{T}^\ast \mathcal{F}_{n}$. 

Now we turn to some examples:

\subsection{
    Magnetic geodesics on $\CP^1$.}

We start with the e.o.m. for the system (\ref{ActionSpin}) on ${\CP^1 \times \CP^1}$
\begin{align}\label{magnetGeodSphere}
    i\, \dot{z}_1 - \alpha_{12}\Bar{a}\, z_2 +\lambda_1 z_1 = 0,\\
    i\, \dot{z}_2 - \alpha_{12} a\, z_1 +\lambda_2 z_2 = 0,\nonumber 
\end{align}
where $a := \Bar{z}_1 z_2$ and $\lambda_1, \lambda_2$ are Lagrange multipliers imposing the constraints $\Bar{z}_i z_i = p_i$, which generate gauge transformations
\begin{align}
    z_1 \mapsto e^{i\varphi_1(t)} z_1\,, \quad\quad z_2 \mapsto e^{i\varphi_2(t)} z_2\,,\label{gaugeTransfSU2}
\end{align}
where $\varphi_1(t), \varphi_2(t)$ are arbitrary real functions. At the same time $\lambda_1, \lambda_2$ play the role of gauge fields and under (\ref{gaugeTransfSU2}) they transform as
\begin{align}
    \lambda_1 \mapsto \lambda_1 + \dot{\varphi}_1\,, \quad \quad \lambda_2 \mapsto \lambda_2 + \dot{\varphi}_2
\end{align}
We note that $|a|^2 = \mathrm{const}$. One can check that $a = \mathrm{const}$ if one chooses the gauge 
\begin{align}
    \lambda_1 = - \alpha_{12} p_1\,,\quad\quad
    \lambda_2 = - \alpha_{12} p_2\,.
\end{align}
It is now easy to find the solution to (\ref{magnetGeodSphere}):
\begin{align} \label{geodCP1CP1}
    \mathsf{Z}(t) := \begin{pmatrix}
        z_1 & z_2
    \end{pmatrix} = \mathsf{Z}(0) \times \exp\left[-i\,\alpha_{12}
    \begin{pmatrix}
        p_1 & a\\
        \Bar{a} & p_2
    \end{pmatrix}t\right]\,.
\end{align}
We now map the non-degenerate $\mathsf{Z}$ (assuming $\mathsf{Z}(0)$ non-degenerate) to the unitary matrix $\mathsf{U} = \mathsf{Z} \left(\mathsf{Z}^{\dagger}\mathsf{Z}\right)^{-1/2}$
\begin{align} \label{geodSphere}
    \mathsf{U}(t) = \mathsf{U}(0) \times \exp\left[-i\,\alpha_{12}
    \begin{pmatrix}
        p_1 & a\\
        \Bar{a} & p_2
    \end{pmatrix}t\right]\,.
\end{align}
The only difference between (\ref{geodCP1CP1}) and (\ref{geodSphere}) is the substitution $\mathsf{Z}(0) \mapsto \mathsf{U}(0)$. In~(\ref{geodSphere}) one recognizes the well-known magnetic geodesic\footnote{Magnetic geodesics on a sphere were widely studied in the literature, cf.~\cite{Kordyukov2020}. The general answer for magnetic geodesics on homogeneous spaces with normal metrics can be found in~\cite{Bolsinov2006}.} on the sphere $\mathcal{S}^2 = \CP^1$ with the usual normal metric $\mathrm{d}s^2 = 1/\alpha_{12}\,\Big(\mathrm{d}\Bar{u}_1 u_2\Big) \Big(\Bar{u}_2 \mathrm{d}u_1\Big)$. One can also explicitly check that~(\ref{geodSphere}) solves the e.o.m. following from (\ref{finalAction}). However, as has been previously noted, the parameter $a$ in (\ref{geodSphere}) is not arbitrary, satisfying the condition $|a|^2 < p_1 p_2$. Nevertheless, it does not spoil the conclusion. Indeed, if we shift $p_i \mapsto p_i + 1$, the solution is multiplied by $e^{-i\alpha_{12}t} \times \mathds{1}_2$, 
which is simply a gauge transformation. The observation is also compatible with Proposition \ref{magneticProp}: the considered change of $p_i$'s leaves $\Omega_{\text{M}}^0$ intact due to $i\,\mathrm{Tr}\left(\mathrm{d}\mathsf{U}^\dagger \wedge \mathrm{d}\mathsf{U}\right) = 0$. Thus, every magnetic geodesic on $\CP^1$ can be found as a solution to the variational problem for the `spin system' (\ref{ActionSpin}) on $\CP^1\times \CP^1$ for sufficiently large $p_1,p_2$ with fixed $q=p_2-p_1$.

\subsection{
    Magnetic geodesics on $\mathcal{F}_3$ and $\mathcal{F}_n$.}

We start with the flag manifold $\mathcal{F}_3$ and the e.o.m. for the system (\ref{ActionSpin}) on $\CP^2\times \CP^2 \times \CP^2$
\begin{align}
    i\, \dot{z}_1 - \alpha_{12}\Bar{a}\, z_2 - \alpha_{13}\Bar{b}\,z_3 + \lambda_1 z_1 = 0,\nonumber\\
    i\, \dot{z}_2 - \alpha_{12} a\, z_1 - \alpha_{23} \Bar{c}\,z_3 + \lambda_2 z_2 = 0,\label{CP2systemEquat}\\
    i\, \dot{z}_3 - \alpha_{13} b\,z_1 - \alpha_{23} c\, z_2 + \lambda_3 z_3 = 0,\nonumber 
\end{align}
where $a = \Bar{z}_1 z_2,\,b=\Bar{z}_1 z_3,\,c=\Bar{z}_2 z_3$ and the constraints $\Bar{z}_i z_i = p_i$ are implied. As we learned above, this system corresponds to the following metric on $\mathcal{F}_3$:
\begin{align}
    \mathrm{d}s^2 = \frac{1}{\alpha_{12}} \Big(\mathrm{d} \Bar{u}_1 u_2 \Big)\Big(\Bar{u}_2 \mathrm{d} u_1\Big) + \frac{1}{\alpha_{23}} \Big(\mathrm{d} \Bar{u}_2 u_3\Big) \Big(\Bar{u}_3 \mathrm{d} u_2\Big) + \frac{1}{\alpha_{13}} \Big(\mathrm{d} \Bar{u}_1 u_3\Big) \Big(\Bar{u}_3 \mathrm{d} u_1\Big)\,. 
\end{align}

The system of equations (\ref{CP2systemEquat}) is a complex system of nonlinear equations. However, in the non-magnetic case, an explicit solution to (\ref{CP2systemEquat}) can be found when $\alpha_{13} = \alpha_{23}$ \cite{Souris,Bykov_2024}. Our goal here is to extend this to the magnetic case.

When $\alpha_{13} = \alpha_{23}$, the equations on $a,b,c$ read:
\begin{align}
    &i\,\dot{a} - \alpha_{12} (p_1-p_2) a - (\lambda_1 - \lambda_2) a = 0\,,\nonumber\\
    &i\,\dot{b} - \alpha_{13} (p_1-p_3) b - \left(\alpha_{13} - \alpha_{12}\right)ac - (\lambda_1 - \lambda_3) b = 0\,,\\
    &i\,\dot{c} - \alpha_{13} (p_2-p_3) c - \left(\alpha_{13} - \alpha_{12}\right)\Bar{a}b - (\lambda_2 - \lambda_3) c = 0\,.\nonumber
\end{align}
As in the $\SU(2)$ case above one can fix $\lambda_1 = -\alpha_{12} p_1,\,\lambda_2 = -\alpha_{12}p_2,\,\lambda_3 = - \alpha_{13} p_3 $, using the gauge transformation $z_j \mapsto e^{i\varphi_j(t)}z_j$, generated by the constraints on $z_j$'s. Then the equation on $a$ takes the especially simple form $\dot{a} = 0$, so that $a = a_0 = \mathrm{const}$. We also explicitly solve the equations on $b,c$:
\begin{align}
    \begin{pmatrix}
        b \\
        c
    \end{pmatrix}(t) = \exp 
    \left[
    -i\,\left(\alpha_{13} - \alpha_{12}\right)
    \begin{pmatrix}
        p_1 & a_0 \\
        \Bar{a}_0 & p_2
    \end{pmatrix}t
    \right]
    \times
    \begin{pmatrix}
        b_0 \\
        c_0
    \end{pmatrix} = \mathrm{G}(t) \times
    \begin{pmatrix}
        b_0 \\
        c_0
    \end{pmatrix}\,.
\end{align}
We will now return to the original equations (\ref{CP2systemEquat}) and solve them. To this end, let us rewrite the equations in terms of the matrix $\mathsf{Z}:=\begin{pmatrix}
    z_1 & z_2 & z_3
\end{pmatrix}$:
\begin{align}
    i\,\dot{\mathsf{Z}} = \mathsf{Z} \times 
    \begin{pmatrix}
        \alpha_{12}p_1 & \alpha_{12}a & \alpha_{13}b \\
        \alpha_{12}\Bar{a} & \alpha_{12}p_2 & \alpha_{13}c \\
        \alpha_{13}\Bar{b} & \alpha_{13}\Bar{c} & \alpha_{13}p_3
    \end{pmatrix}\,.\label{F123geodEq}
\end{align}
To solve (\ref{F123geodEq}), consider the matrix 
\begin{align}
    \mathsf{Y}(t) = \mathsf{Z}(t) \times \begin{pmatrix}
        \mathrm{G}(t) & \\
        & 1
    \end{pmatrix}\,.
\end{align}
The equation on $\mathsf{Y}$ is as follows:
\begin{align}
    i\,\dot{\mathsf{Y}} = \mathsf{Y} \times \alpha_{13}
    \begin{pmatrix}
        p_1 & a_0 & b_0 \\
        \Bar{a}_0 & p_2 & c_0 \\
        \Bar{b}_0 & \Bar{c}_0 & p_3
    \end{pmatrix}\,,
\end{align}
which is an equation with constant coefficients. Solving it, we find
\begin{align}
    &\label{spinchainF3}\mathsf{Z}(t) =\mathsf{Z}(0) \times \mathrm{Ev}(t)\,,\quad \text{where}\,\, \\
    &\mathrm{Ev}(t) := \exp\left[-i\alpha_{13}\begin{pmatrix}
        p_1 & a_0 & b_0 \\
        \Bar{a}_0 & p_2 & c_0 \\
        \Bar{b}_0 & \Bar{c}_0 & p_3
    \end{pmatrix} t \right]\times 
    \begin{pmatrix}
        \exp 
        \left[
        i\,\left(\alpha_{13} - \alpha_{12}\right)
        \begin{pmatrix}
            p_1 & a_0 \\
            \Bar{a}_0 & p_2
        \end{pmatrix}t
        \right] & \\
        & 1
    \end{pmatrix}\,.\nonumber
\end{align}
As in the case of $\CP^1$, for non-degenerate $\mathsf{Z}$ one can construct the unitary matrix $\mathsf{U}=\mathsf{Z}\left(\mathsf{Z}^{\dagger}\mathsf{Z}\right)^{-1/2}$ and interpret it as a magnetic geodesic on $\mathcal{F}_3$
\begin{align}
    \mathsf{U}(t) = \mathsf{U}(0)\times \mathrm{Ev}(t)\,, \label{geodF123}
\end{align}
which coincides with the result of the work \cite{ArvanitoyergosSouris}. We also observe that under ${p_i \mapsto p_i + 1}$ the mere change in (\ref{geodF123})  is by a gauge transformation.

One can now easily generalize to the case of $\mathcal{F}_n$ by considering the metric
\begin{align}
    \mathrm{d}s^2 = \sum\limits_{k=2}^n \frac{1}{\alpha_k} \sum\limits_{j=1}^{k-1} \Big(\mathrm{d}\Bar{u}_k u_j \Big) \Big(\Bar{u}_j \mathrm{d}u_k\Big)\,,
\end{align}
where $u_k$'s are columns of the matrix $\mathsf{U}$ parameterizing $\mathcal{F}_n$. The corresponding system on $\left(\CP^{n-1}\right)^{\times n}$ is defined by (\ref{ActionSpin}) with $\alpha_{ij} := \alpha_{\mathrm{max}(i,j)}$ and the constraints $\Bar{z}_i z_i = p_i$. The e.o.m. for the system take the form
\begin{align}
    i\,\dot{z}_a - \sum\limits_{i<a}\alpha_a \left(\Bar{z}_i z_a\right) z_i - \sum\limits_{a < i}\alpha_i \left(\Bar{z}_i z_a\right) z_i + \lambda_a z_a = 0\,,
\end{align}
where $a = 1,2,\dots,n$ and $\alpha_1 := \alpha_2$. We fix the gauge by choosing $\lambda_a = -\alpha_a p_a$ and define the matrices 
\begin{align}
    \left(\mathcal{A}\right)_{ab} := \Bar{z}_a z_b,\quad \left(\mathcal{B}\right)_{ab} := \alpha_{\mathrm{max}(a,b)} \,\Bar{z}_a z_b\,.
\end{align}
The equations on $\Bar{z}_a z_b$'s can be written in the compact form
\begin{align}\label{magnCPn-1}
    i\,\dot{\mathcal{A}} + \left[\mathcal{B},\mathcal{A}\right] = 0\,.
\end{align}
Equations of the same type were studied in \cite{Bykov_2024} and the same analysis can be applied to the present situation without any technical difficulties. As a result one solves the e.o.m. and, consequently, finds geodesics using the above reasoning. We present an answer for $\mathsf{Z}(t)$ and the magnetic geodesics on $\mathcal{F}_n$ in a uniform way via the evolutionary operator $\mathrm{Ev}(t)$, which defines their dynamics as in (\ref{spinchainF3}) and (\ref{geodF123}):
\begin{align}
    \mathrm{Ev}(t) = \overleftarrow{\prod\limits_{k=2}^{n}} \exp\left[i\left(\alpha_{k+1} - \alpha_{k}\right)t\,\mathrm{Pr}_{k}\left(\mathcal{A}^0\right)\right]\,,
\end{align}
where $\alpha_{n+1} := 0$, $\mathcal{A}^0 := \mathcal{A}(0)$, $\mathrm{Pr}_k$ is the projector on the upper left  $(k\times k)$ block of the matrix and the arrow means multiplication of the matrices from right to left\footnote{For example, $\overleftarrow{\prod\limits_{k=1}^{n}}\mathcal{C}_k = \mathcal{C}_n\dots\mathcal{C}_2\,\mathcal{C}_1$.}. 

We have discussed a special family of metrics and found their magnetic geodesics on complete flag manifolds. However, we expect that similar considerations can be carried out for general metrics with `nested structure' on arbitrary flag manifolds of $\mathrm{SU}(n)$ (see \cite{Bykov_2024} for the definition).

\vspace{0.5cm}
\textbf{Acknowledgments.} Sections 1-3 were written with the support of the Foundation for the Advancement of Theoretical Physics and Mathematics ``BASIS''. Sections 4-6 were supported by the Russian Science Foundation grant № 20-72-10144 (\href{https://rscf.ru/en/project/20-72-10144/}{\emph{https://rscf.ru/en/project/20-72-10144/}}). We would like to thank S.~Gorchinskiy, I.~Taimanov and N.~Tyurin for inspiring discussions.

\vspace{1cm}
\appendix
\appendixbig

\setcounter{section}{10}
\newcounter{appcounter}
\setcounter{appcounter}{1}
\renewcommand{\thesection}{\Alph{appcounter}}

\section{
    Regular values of the moment map.}\label{regularappendix}

A well-known statement asserts that if $0$ is a \emph{regular}
 value of the moment map, then $\mu^{-1}(0)$ is Lagrangian \emph{if and only if} it is a $\mathrm{G}$-orbit~\cite[Lemma 2.4.2]{Audin}. Our goal here is to show how this principle can break down whenever zero is not a regular value.

 As the simplest example, consider the manifold $\mathcal{M}_3:=\CP^1 \times \CP^1 \times \CP^1$ with the symplectic form
 \bea
\Omega=\alpha\, \Omega_1+\beta\, \Omega_2+\gamma\,  \Omega_3\,,
 \eea
 where $\alpha, \beta, \gamma>0$ and $\Omega_{1,2,3}$ are copies of the Fubini-Study form on the three spheres. The moment map for the action of $\tSO(3)$ may be thought of as a vector in $\mathbb{R}^3$:
 \begin{align}
\vec{\mu}=\alpha\,\vec{n}_1+\beta\,\vec{n}_2+\gamma\,\vec{n}_3\,.
 \end{align}
 The set $\vec{\mu}=0$ defines a triangle in $\mathbb{R}^3$ with sides of length $\alpha, \beta, \gamma$, which exists whenever the triangle inequalities on $\alpha, \beta, \gamma$ are satisfied. If these values are generic (i.e., none of the inequalities is saturated), $0$ is the regular value of $\vec{\mu}$, and $\vec{\mu}^{-1}(0)\simeq \tSO(3)$. It is clearly a Lagrangian submanifold of $\mathcal{M}_3$. However in the exceptional case when one of the triangle inequalities is saturated, say $\gamma=\alpha+\beta$, equation $\vec{\mu}=0$ implies that $\vec{n}_1=\vec{n}_2=-\vec{n}_3$, so that $\vec{\mu}^{-1}(0)\simeq \CP^1$ is only isotropic.

\vspace{1cm}    
    \setstretch{0.8}
    \setlength\bibitemsep{5pt}
    \printbibliography
\end{document}